\documentclass[10pt]{article}
\usepackage{bbm}
\usepackage{amssymb,amsmath,amsfonts}
\usepackage{dsfont}
\usepackage{amsthm}
\usepackage{graphicx,color,subfig}
\usepackage{xspace}
\usepackage[latin1]{inputenc}
\usepackage[T1]{fontenc}
\usepackage{mathrsfs}
\usepackage{amscd}
\usepackage[english]{babel}
\usepackage[active]{srcltx}

\let \Re \relax
\DeclareMathOperator{\Re}{Re}
\let \Im \relax
\DeclareMathOperator{\Im}{Im}
\newcommand{\Con}{\ensuremath{\mathscr C}}
\newcommand{\Cinf}{\ensuremath{\Con^\infty}}
\renewcommand{\S}{\ensuremath{\mathscr S}}

\DeclareMathOperator{\supp}{supp}

\newcommand{\mb}[1]{\ensuremath{\mathbb{#1}}}
\newcommand{\N}{{\mb{N}}}

\newcommand{\R}{{\mb{R}}}

\renewcommand{\d}{\ensuremath{\partial}}

\DeclareMathOperator{\Op}{Op}

\DeclareMathOperator{\Div}{div}

\newcommand{\nhd}{neighborhood\xspace}

\makeatletter
\def\keywords{
    \vspace{1ex}
    \noindent
    \if@twocolumn
      \small{\bf  Keywords}\/---$\!$    \else
      \begin{center}\small\ {\bf Keywords}\end{center}\quotation\small
    \fi}
\def\endkeywords{\vspace{0.6em}\par\if@twocolumn\else\endquotation\fi
    \normalsize\rm}
\makeatother

\newtheorem{theorem}{Theorem}[section]
\newtheorem{lemma}{Lemma}[section]
\newtheorem{remark}{Remark}[section]

\newtheorem{definition}{Definition}[section]

\begin{document}

\date{\today}
\title{Logarithmic Decay of a Wave Equation with Kelvin-Voigt Damping
\thanks{
This work was supported by the National Natural Science Foundation of China (grants No. 60974033) and Beijing Municipal Natural Science Foundation (grant No. 4182059).
\medskip} }

\author{Luc Robbiano
 \thanks{Laboratoire de Math\'ematiques Appliqu\'ees, UMR 8100 du CNRS, Universit\'e Paris--Saclay (site UVSQ),
45 avenue des Etats Unis, 78035 Versailles Cedex, France (luc.robbiano@uvsq.fr)}
 and Qiong Zhang
 \thanks{
School of Mathematics and Statistics, Beijing Key Laboratory on MCAACI, Beijing Institute of Technology, Beijing, 100081, China
(zhangqiong@bit.edu.cn).}
\thanks{Corresponding author.}}


\maketitle

\begin{abstract}
In this paper we analyze   the long time behavior of a wave equation with local Kelvin-Voigt Damping.
Through  introducing proper class
symbol and pseudo-differential calculus, we obtain a Carleman estimate, and then
establish an estimate on the corresponding resolvent operator. As a result,
we show the logarithmic decay rate for energy of the system without
any geometric assumption on the subdomain on which the damping is effective.

\vskip 4mm

\noindent
{\bf keywords.}
   Carleman estimate; wave equation; Kelvin-Voigt damping; logarithmic stability

\vskip 4mm

\noindent
  {\bfseries AMS 2010 subject classification:}
Primary 93B05; Secondary 93B07, 35B37

\end{abstract}


%
%

\section{Introduction}

In this paper, we consider a wave equation with local Kelvin-Voigt damping
and analyze long time behaviour for the -solution of the system.
Let $\Omega\subset\mathbb{R}^d $ be a bounded domain with smooth boundary $\Gamma = \partial \Omega$.
  Denote by   $\partial_{n}$  the unit outward normal vector  on boundary $\Gamma.$ The PDE model is as follows.
 \begin{equation}\label{system}
\left\{
\begin{array}{lcl}
  y_{tt}(t,x) -   \Div\, [\nabla y (t,x)+ a(x) \nabla    y_{t}(t,x)]  = 0  & \mbox{ in } & (0, \infty) \times \Omega, \\ \noalign{\medskip}
 y(t,x) = 0& \mbox{ on } & (0, \infty) \times \Gamma, \\ \noalign{\medskip}
  y(0,x)  = y^0, \quad   y_{t}(0,x)  = y^1 & \mbox{ in } & \Omega,
\end{array}
\right.
\end{equation}
where the coefficient function $a(\cdot)\in L^1( \Omega)$~
 is   nonnegative  and not identically null.

The natural energy of system (\ref{system}) is
\begin{equation}
 \label{energy}
E(t)=\frac{1}{2}
\Big[\int_{\Omega}\big(\,|\,\nabla y(t)\,|\,^2+\,|\,  y_t(t)\,|\,^2\big) dx  \Big].
 \end{equation}

\noindent
A direct computation gives that
\begin{equation}\label{energydiss}
{{d}\over{dt}}E(t) = -\int_{\supp a} a(x)\,|\,\nabla
y_t(t) \,|\,^2 dx.
\end{equation}
Formula (\ref{energydiss}) shows that the only dissipative mechanism acting on the system is the viscoelastic
damping $\Div\,[a\nabla y_t] $, which is only effective on $\supp a$.

To rewrite the system as an evolution equation, we set the energy space as
\begin{equation}
\label{space}
 {\cal{H}} =
   H^1_0(\Omega) \times L^2(\Omega),
\end{equation}
with norm
\begin{equation}
\label{norm}
 \|Y\|_{\cal{H}}  =   \sqrt{\|\nabla y_1\|_{L^2(\Omega)}^2
 +\| y_2\|_{L^2(\Omega)}^2 },\;\;\;\;
   \forall \;Y = \big(y_1,y_2\big) \in \cal{H}.
 \end{equation}

\noindent
Define an unbounded operator ${\cal A}: D({\cal A}) \subset
{\cal{H}}\rightarrow {\cal{H}}$ by
$$
{\cal A} Y=  \big(y_2,\Div\,(\nabla y_1 + a \nabla y_2) \big),  \quad \forall\; Y
= \big(y_1,y_2\big) \in D({\cal A}),
$$
and
$$
{\cal D}({\cal A} )  =
 \Big\{\big(y_1,y_2\big)\in {\cal{H}} \: : \:
  y_2 \in H^1_0(\Omega) ,
 \Div\,(\nabla y_1 + a \nabla y_2) \big)\in L^2(\Omega) \Big\}.
$$

\noindent Let $Y(t) = (y(t),\, y_t(t))$.  Then system (\ref{system}) can be written as
\begin{equation}
\label{system1}
{{d}\over{dt}} Y(t)  ={\cal A}Y(t)
, \quad \forall \;t> 0, \quad    Y(0) = (y^0,y^1).
\end{equation}
%

\indent
It is known from  \cite{liurao} that if   $\supp a $ is non-empty, the operator  ${\cal A}$ generates a contractive
 $C_0$ semigroup $e^{t{\cal A}}$ on ${\cal H}$ and $i{\mathbb R}\subset\rho({\cal A})  $, the resolvent of ${\cal A}$.
 Consequently,  the semigroup $e^{t{\cal A}}$ is strongly stable.
 Moreover, if the entire medium is of the viscoelastic type (i.e. $\supp a = \overline{\Omega}$),  the
damping for the wave equation not only induces exponential energy
decay, but also restricts the spectrum of the associated semigroup
generator to a sector in the left half plane, and the associated
semigroup is analytic (\cite{huang1}).
When the Kelvin-Voigt damping is localized on a subdomain of $\Omega$, the properties of system is quite complicated. 
First, it has been proved that properties of  regularity and stability of 1-d system \eqref{system}
depend  on the continuousness of  coefficient function $a(\cdot)$.
More precisely, assume that $\Omega = (-1,\,1)$ and  $a(x)$ behaviours  like $ x^\alpha $ with $\alpha >0 $ in $\supp a= [0,1]$.
Then the solution of system \eqref{system} is eventually differentiable for $\alpha >1$, exponentially stable for $\alpha \ge 1$, polynomially stable of order $\frac{1}{1-\alpha }$ for $0<\alpha <1$, and polynomially stable of optimal decay rate  $2$ for $ \alpha =0 $
 (see \cite{liu-liu-qiong, liu-qiong,renardy, qiong-2010}).
For the higher dimensional system,
the corresponding semigroup  is  exponentially stable when $a(\cdot)\in \Con^2(\Omega)$   and $ \supp a \supset \Gamma$  (\cite{liu-rao}).
However, when the Kelvin-Voigt damping is local
 and the material coefficient $a(\cdot)$ is a positive constant on $\supp a$,   the energy of system (\ref{system}) does not
 decay exponentially for any geometry of $\Omega$ and $\supp a$ (\cite{cll,qiong2017}). The reason is that the strong damping and   non-continuousness of the coefficient function lead to
reflection of waves at the interface $\gamma \doteq \partial (\supp a) \setminus \Gamma$, which then fails to be effectively damped because they do not enter the
region of damping.
It turns out that  the Kelvin-Voigt damping does not follow the principle that
``geometric optics" condition implies exponential stability, which is true for the wave equation with local viscous damping (\cite{bardos}).

Recently,  \cite{qiong2018} proves the polynomial stability of system \eqref{system} when $a(\cdot)\equiv a_0>0$  on $\supp a $ and $\supp a $ satisfies certain geometry conditions.
Then, a   natural problem is: how about the    decay rate if $\supp a\not= \emptyset$ is arbitrary?  
In~\cite{AHR:2018}, $a$ is assumed discontinuous along a $(d-1)$--manifold, $\supp a$ is arbitrary and the rate of the decay of semi-group is estimated
by $(\log t)^{-k}$ for a
data in $D({\cal A}^k)$.
 In this paper, we analyze the logarithmic
decay properties of the solution to \eqref{system} when $a$ is smooth and $\supp a $ is arbitrary.
The main result   reads as follows.
\begin{theorem}\label{th2}
Suppose that the coefficient function $a(\cdot)\in \Con_0^\infty( \Omega)$~ is   nonnegative and $\supp a \subset \Omega$ is non-empty.
 Then the energy of the solution of (\ref{system})
decays at logarithmic speed. More precisely, one has that  there exists  a positive constant $C$ such that
\begin{equation}
\label{th-energy-decay}
\|e^{t{\cal A}} Y_0 \|_{\cal H} \le {{C}\over{[\log(t+2)]^{4k\over5}}}\Big\|Y_0\Big\|_{D({\cal A}^k)},  \quad  \forall \;t>0, \;\; Y_0=(y^0,y^1)\in D({\cal A}^k).
\end{equation}
\end{theorem}

Our approach is based on the results duo to \cite{burq}, which reduced the problem of determining the rate of energy decay to estimating the norm of the resolvent operator along the imaginary axis, see also \cite{Duyckaerts, lebau-robbiano1}, etc.
Our argument divides naturally three steps. First, in Section 2, we show some preliminaries including definitions and classical results about symbol, pseudo-differential calculus, 
and commutator estimate, etc.  
Then in Section 3, we prove corresponding Carleman estimates. Finally, in Section 4, we present a resolvent estimate and obtain Theorem \ref{th2}.
This theorem is a consequence of a resolvent estimate. The proof is given in Section~\ref{section: resolvent estimate}.
The method was  developed in    (\cite{Bellassoued1, Bellassoued2, Bellassoued3, lebau, lebau-robbiano1, robbiano} and the references cited therein).

Throughout this paper, we use $\|\cdot\|_V$ and $( \cdot\,|\,\cdot)_V$ to denote the norm and inner product on $L^2(V)$, where $V\subset {\mathbb R}^d$ if there is no further comments. When writing
$f\lesssim g$ (or $f \gtrsim g $), we mean that there exists a positive constant $C$ such that $f \le Cg$ (or $f\ge Cg$).
For $j=1,2,\cdots, $   define  operators   $D_j=-i\partial_{x_j}$,
$D = (D_1,\cdots, D_d)$, $D^2=\sum\limits_{j=1}^dD_j^2$  and
$Da(x)D=\sum\limits_{j=1}^dD_ja(x)D_j$.

\section{Preliminaries}
\setcounter{equation}{0} \setcounter{theorem}{0}


%
%
We shall prove Theorem \ref{th2} by Weyl-H\"ormander calculus, which was introduced H\"ormander (\cite{Hormander83,Rousseau}).
In this section,  some definitions and results on the class of symbol and pseudo-differential calculus are given.

\subsection{Symbol and Symbolic calculus}
\label{Sec: Symbol}


 For any $(x,\xi) \in \R^d \times\R^d, \; \lambda\in \R$ and $  \tau>0 $, we introduce the metric
 \begin{equation}
 \label{eq:metric}
 g=g_{x,\xi}=\lambda dx^2+\mu^{-2}d\xi^2, \quad \;  \mbox{ where  }  \;  \mu^2=\mu(\tau,\xi)^2=\tau^2+\,|\,\xi\,|\,^2,
 \end{equation}
 and the weight
 \begin{equation}
 \label{eq:weight}
\nu =\nu(x,\lambda) =\sqrt{1+\lambda^2a(x)^2}.
 \end{equation}
Note that and  $g_{x,\xi}(X,\Xi)=\lambda \,|\,X\,|\,^2+\mu^{-2}(\tau,\xi)\,|\,\Xi\,|\,^2$
 for all $X,\Xi\in\R^d$.
 Then we have the following results.
%

%
%
%
%
%
\begin{lemma}\label{lemma-metric-weight}
Assume that there exist positive constants $\,{\cal C}$ and $ \lambda_0$ such that $\lambda\ge \lambda_0$ and $\tau \ge \max\{\,{\cal C}  \lambda , \; 1\,\} $. It holds
\begin{enumerate}
\item[{\rm (i)}]The metric  $g=g_{x,\xi} $ defined by \eqref{eq:metric} is
admissible, i.e., it is slowly varying and temperate.
\item[{\rm (ii)}] The weight   $ \nu = \nu(x,\lambda)$  defined by \eqref{eq:weight} is   admissible, i.e., it is $g$-continuous and $g$-temperate.
\end{enumerate}
\end{lemma}

\begin{proof} {\bf  (i) }
From Definition 18.4.1 in \cite{Hormander83},    the   metric  $g_{x,\xi}$ defined by \eqref{eq:metric} is  {\bf slowly varying} if there exist $\delta>0$ and $C>0$ such that
$$
g_{x,\xi}(y-x,\eta-\xi)\le \delta \textrm{ implies } g_{y,\eta} (X,\Xi)   \le C  g_{x,\xi} (X,\Xi),\; \; \forall\;x,y,\xi,\eta, X,\Xi\in\R^d,
$$
where the constants $\delta$ and $C$ are independent on the parameters $\lambda$ and $\tau$.

Suppose  $0<\delta  \le 1/4$ and
$$
g_{x,\xi}(y-x,\eta-\xi)=\lambda \,|\,y-x\,|\,^2+(\tau^2+\,|\,\xi\,|\,^2)^{-1}\,|\,\eta-\xi \,|\,^2\le\delta.
$$
Then, we have
$$
\begin{array}{lcl}
\tau^2+\,|\, \xi \,|\,^2 & \le& \tau^2+2\,|\, \xi -\eta\,|\,^2+2\,|\, \eta \,|\,^2
\\ \noalign{\medskip} &  \le & \tau^2+2 \delta (\tau^2+\,|\,\xi\,|\,^2) +2\,|\, \eta \,|\,^2.
\end{array}
$$
This implies that $\tau^2+\,|\, \xi \,|\,^2 \le 4( \tau^2 +\,|\, \eta \,|\,^2)$. Consequently,
$$
g_{y,\eta} (X,\Xi) =\lambda \,|\,X\,|\,^2+( \tau^2 +\,|\, \eta \,|\,^2)^{-1}\,|\,\Xi\,|\,^2\le \lambda \,|\,X\,|\,^2+ {1\over 4}(\tau^2+\,|\, \xi \,|\,^2)^{-1}\,|\,\Xi\,|\,^2\le     g_{x,\xi} (X,\Xi) .
$$
Therefore, $g$ is slowly varying.

For a  given metric  $g_{x,\xi}$, the  associated metric $g_{x,\xi}^\sigma$ is defined by
  $g_{x,\xi}^\sigma=(\tau^2+\,|\, \xi \,|\,^2) dx^2+\lambda^{-1}d\xi^2$.
The metric $g_{x,\xi}$  is {\bf temperate } if there exist $C>0$ and $N>0$, such that
\begin{equation}
\label{eq: temperate}
g_{x,\xi}(X,\Xi)\le C g_{y,\eta}(X,\Xi)\big(1+g^\sigma_{x,\xi}(x-y,\xi-\eta)\big)^N, \;\; \forall\; x,y,\xi,\eta, X,\Xi\in\R^d,
\end{equation}
where the constants $C$ and $N$ are independent on  the parameters $\lambda$ and $\tau$ (Definition 18.5.1 in \cite{Hormander83}).

\indent
For  the metric $g=g_{x,\xi}$ defined by \eqref{eq:metric}, \eqref{eq: temperate} is equivalent to
\begin{equation}
\label{eq: temperate1}
\begin{array}{ll}&
\lambda \,|\,X\,|\,^2+(\tau^2+\,|\, \xi \,|\,^2)^{-1}\,|\,\Xi\,|\,^2
\\ \noalign{\medskip} \le& C \big( \lambda \,|\,X\,|\,^2+(\tau^2+\,|\, \eta \,|\,^2)^{-1}\,|\,\Xi\,|\,^2\big)\big(1+(\tau^2+\,|\, \xi \,|\,^2)\,|\,x-y\,|\,^2
+\lambda^{-1}\,|\,\xi-\eta\,|\,^2\big)^N.
\end{array}
\end{equation}

\noindent
First, assume that
  $\tau^2 + \,|\,\eta\,|\,^2\le 4(\tau^2+\,|\,\xi\,|\,^2)$. It follows that
\begin{equation} \label{eq: metric temperate}
 (\tau^2+\,|\, \eta \,|\,^2)\le C(\tau^2+\,|\, \xi \,|\,^2) \big(1+\lambda^{-1}\,|\,\xi-\eta\,|\,^2\big)^N, \quad   C>0,\;N>0.
\end{equation}
Then it is easy to obtain \eqref{eq: temperate1} from \eqref{eq: metric temperate}.

Secondly, consider the case
  $\tau^2 + \,|\,\eta\,|\,^2 > 4(\tau^2+\,|\,\xi\,|\,^2)$.
  Then
\begin{equation}
\label{eq: metric t1}
\,|\,\eta\,|\,  > 2\,|\,\xi\,|,   \quad   \,|\,\eta\,|\, > \sqrt{3}\tau,
\end{equation}
 and
\begin{equation}
\label{eq: metric t2}
\,|\,\xi-\eta\,|\,> {1\over2}\,|\,\eta\,|\,> {\sqrt{3}\over2}\tau> {\sqrt{3}\over2}{\cal C}\lambda.
\end{equation}
 It follows from \eqref{eq: metric t1} and \eqref{eq: metric t2} that
  $$
  \lambda^{-1}\,|\,\xi-\eta\,|\,^2
  >{\sqrt{3}\over2}{\cal C} \,|\,\xi-\eta\,|\,
  >  {\sqrt{3}\over4}{\cal C}\,|\,\eta\,|\,.
  $$
  Consequently,
  $$
    \big(1+\lambda^{-1}\,|\,\xi-\eta\,|\,^2\big)^2 >  {3\over16}{\cal C}^2\,|\,\eta\,|\,^2 >{3\over32}{\cal C}^2(\,|\,\eta\,|\,^2+ 3\tau^2) .
  $$
  This together with $\tau^2+\,|\, \xi \,|\,^2 \ge1$ yields that  there exists a positive constant $C$ such that  ~\eqref{eq: metric temperate}  holds with $N=2$.

 {\bf   (ii) } It is known from  Definition 18.4.2   in \cite{Hormander83} that a weight $\nu(x)$ is {\bf  $g$-continuous} if there exist $\delta>0$ and $C>0$ such that
$$
g_{x,\xi}(y-x,\eta-\xi)\le\delta\; \textrm{ implies }\;  C^{-1}\nu(x)\le \nu(y)\le C \nu(x),\; \; \forall\; x,y,\xi,\eta\in\R^d.
$$
where the constants $\delta$ and $C$ are independent on the parameters $\lambda$ and $\tau$.
 Since the weight  $\nu(x)$ defined by \eqref{eq:weight} does not depend on $\xi$, the above condition  is reduced to
$$\lambda\,|\,x-y\,|\,^2\le\delta\; \textrm{ implies }\;  C^{-1} \nu(x) \le \nu(y)\le C \nu(x),\; \; \forall\; x,y \in\R^d.$$

\noindent
The weight $\nu(x)$ is  {\bf $g$-temperate} if there exist $C>0$ and $N>0$ such that
\begin{equation}
\label{eq:def-s-temperate}
\nu(y)\le C\nu(x)\big(1+g^\sigma_{y,\eta}(x-y,\xi-\eta)  \big)^N,\; \; \forall\; x,y,\xi,\eta\in\R^d,
\end{equation}
where the constants $C$ and $N$ do not depend on the parameters $\lambda$ and $\tau$  (Definition 18.5.1 in \cite{Hormander83}).
The weight $\nu(x)$ is   admissible
 if it is $g$-continuous and $g$-temperate.
When a weight is admissible,
all the powers of this weight are
$g$-continuous and $g$-temperate.
Therefore, it suffice to prove that $ 1+\lambda a(x)$ is admissible.

 Let $s\in[0,t] $ and  $t\in[0,1]$. Define  $f(s)=\lambda a(x+s(y-x))$
and $F(t)=\sup_{s\in[0,t]}f(s)$ where $x,\;y\in \Omega $ satisfying $\lambda\,|\,x-y\,|\,^2 \le \delta$.
It is clear that $f'(s)=\lambda \, a'(x+s(y-x))\,(y-x)$.
Combining this with
the following inequality
\begin{equation}  \label{eq: estimate derivative positive function}
\,|\,a'(x)\,|\,^2\le 2a(x)\|a''\|_{L^\infty(\Omega)}, \;\; \; \forall\;  x\in \Omega,
\end{equation}
we obtain
$$
\,|\,f'(s)\,|\,\le \lambda \,|\,a'(x+s(y-x))\,|\,\,|\,y-x\,|\,\le 2\lambda \| a''\|_\infty^{1\over2}\,[ a( x+s(y-x)) ]^{1\over2}\,|\,y-x\,|\,.
$$
The proof of \eqref{eq: estimate derivative positive function} will be given later.
Consequently,
$$
\sup_{s\in[0,t]}\,|\,f'(s)\,|\,\le 2 \lambda^{1\over2}  \| a''\|_{L^\infty(\Omega)}^{1\over2}\,F(t)^{1\over2}\,|\,y-x\,|\,.
$$
Since $f(t)\le f(0)+t\,\sup_{s\in [0,t]}\,|\,f'(s)\,|\,$,   $F$ is non-decreasing and $\lambda\,|\,x-y\,|\,^2 \le \delta$, we obtain
that
for all $t\in[0,1]$,
$$
f(t)\le f(0)+ C \lambda^{1\over2}  F(t)^{1\over2}\,|\,y-x\,|\,\le  f(0)+ C\sqrt \delta F(t)^{1\over2}\le   f(0)+ C\sqrt \delta F(\alpha)^{1\over2},
$$


\noindent
where $C=2  \| a''\|_{L^\infty(\Omega)}^{1\over2}$ and   $\alpha\in[ t,1]$. Note that $f(0)=F(0)$. It follows that
\begin{equation*}
F(\alpha)=\sup_{t\in [0,\alpha]}f(t)\le F(0)+C\sqrt \delta F(\alpha)^{1\over2}.
\end{equation*}
This yields
\begin{equation}\label{eq: estimation on g(alpha)}
1+F(\alpha)\le 1+ F(0)+C\sqrt \delta\big(1+ F(\alpha) \big)^{1\over2}
\le 1+ F(0)+C\sqrt \delta\big(1+   F(\alpha) \big) .
 \end{equation}
%

 \noindent
By choosing $\delta $ sufficiently small such that $C\sqrt \delta\le 1/2$, one can deduce from ~\eqref{eq: estimation on g(alpha)} that
$$
1+F(\alpha)\le 2( 1+ F(0)),  \quad  \forall\; \alpha \in [t,\,1].
$$
In particular, we have
$$1+\lambda a(y)
\le 2(1+\lambda a(x)).
$$
The above inequality remains true if we   exchange $x$ and $y$. Therefore, the weight $1+\lambda a(x) $ is g-continuous.

 On the other hand,
note that $1+\lambda a(x)$ is independent to $\xi$. Then, to obtain the  weight $1+\lambda a(x) $  is  $\sigma$-temperate,
it is sufficient to prove that
\begin{equation}
\label{eq: m-temperate}
1+\lambda a(y)\le C (1+\lambda a(x) )(1+\tau^2\,|\,x-y\,|\,^2)^N.
 \end{equation}
 In fact,
it is clear that $1+\lambda a(y)\le 1+\lambda ( a(x) +C\,|\,x-y\,|\,)$ where $C=\| a'\|_{L^\infty(\Omega)}$.
Therefore,  there exists positive constant $C'= C {\cal C}^{-1}$ such that
$$
1+\lambda a(y)
\le ( 1+\lambda  a(x))(1+ C' \tau \,|\,x-y\,|\,)
\le ( 1+\lambda  a(x))(2+2(C'\tau \,|\,x-y\,|\,)^2)^{1\over2}.
$$

\noindent
Thus, we obtain \eqref{eq: m-temperate} with $N={1\over2}, \; C= 2\max\{1,\; C'\}$.
 \end{proof}

%
%
%
%

%
%

\vskip 4mm

%
%

\begin{remark}
We claim that \eqref{eq: estimate derivative positive function} holds for any compactly supported and nonnegative  function  $a \in\Con^2(\Omega) $.
In fact, from the following identity
$$
a(x+h) = a(x) + a'(x) h + \int_0^1(1-t) a''(x+th)h^2dt, \;\; \forall \; h\in {\mathbb R},
$$
one can get
$$
  a(x) + a'(x) h +{1\over2} \|a''\|_{L^\infty(\Omega)} \,|\,h\,|\,^2 \ge 0.
$$
Let $h=y\,a'(x),$ where $y\in {\mathbb R}$ and $x\in \Omega$ are arbitrary.
It follows from the above inequality that
$$
  a(x) + \,|\,a'(x)\,|\,^2 y  +{1\over2} \|a''\|_{L^\infty(\Omega)}   \,|\,a'(x)\,|\,^2 y^2 \ge0,  \quad  \forall \; x\in \Omega, \;y\in {\mathbb R}.
$$
Then,
$$
\,|\,a'(x)\,|\,^4 -2 a(x) \|a''\|_{L^\infty(\Omega)}   \,|\,a'(x)\,|\,^2 \le0,
$$
and \eqref{eq: estimate derivative positive function} is proved.
\end{remark}

\vskip 4mm

\begin{definition}(Section 18.4.2 in \cite{Hormander83}) Assume the weight $m(x,\xi)$ is admissible  and the metric $g$ is defined by \eqref{eq:metric}. Let $q(x,\xi, \lambda, \tau)$ be a $ \Cinf$ function with respect to $(x,\,\xi)$ and $\lambda,\; \tau$ be parameters satisfying conditions in Lemma \ref{lemma-metric-weight}.
The symbol $q(x,\xi, \lambda, \tau)$ is in {\bf class
 $S(m,g)$}   if for all $\alpha,\beta\in\N^d$
there exists $C_{\alpha,\beta}$ independent of $\tau$ and $\lambda$ such that
$$
\,|\,\d_x^\alpha\d_\xi^\beta q(x,\xi, \lambda, \tau)\,|\,\le C_{\alpha,\beta} \, m(x,\xi)\,\lambda^{\,|\,\alpha\,|\,/2}\,\mu(\tau, \xi)^{-\,|\,\beta\,|\,}.
$$
\end{definition}

\vskip 4mm

\begin{remark}\label{remark-class1}
\begin{enumerate}
\item[{\rm (i)}]
It is clear that $\mu = \sqrt{\tau^2+\,|\,\xi\,|\,^2 } \in S(\mu,g)$ since
 $
\,|\, \d_\xi^\beta \mu(\tau,\xi) \,|\,\lesssim \mu ^{1-\,|\,\beta\,|\,} $
for all $\beta \in {\mathbb R}^d.$

 \item[{\rm (ii)}] Let $\nu$
be the weight  defined by \eqref{eq:weight}.
It is easy to get that $\lambda a\in S(\nu,g)$. In fact, if  $\,|\,\alpha\,|\,\ge 2$, it holds that $ \,|\, \d_x^\alpha(\lambda a(x))\,|\,\le C_\alpha \lambda\le C_\alpha  \lambda^{\,|\,\alpha\,|\,/2}\nu (x)$ , where   $C_\alpha>0.$
For the case    $\,|\,\alpha\,|\,=1$, it follows from~\eqref{eq: estimate derivative positive function}   that
 $$\,|\,\d^{\alpha}_x(\lambda a(x))\,|\,\le \sqrt{2}\,\|a''\|_{L^\infty(\Omega)}^{1\over2}\lambda^{1\over2} \big( \lambda a(x) \big)^{1\over2} $$
 Note that $\,|\,\lambda a(x)\,|\,<  C \nu^2(x)$ for some $ C>0$.
 This together with the above inequality, we have that
$$\,|\,\d^{\alpha}_x(\lambda a(x))\,|\, < \sqrt{2C}\,\|a''\|_{L^\infty(\Omega)}^{1\over2}  \lambda^{1\over2} \nu(x). $$

\item[{\rm (iii)}]
It is known from Lemma 18.4.3 of \cite{Hormander83}  that  if the metric $g$ and weights $m_1,\;m_2 $  are admissible, symbols $a\in S(m_1,g)$ and $b\in S(m_2,g)$, then $ab\in S(m_1m_2,g)$. In particular,
 $(\lambda a)^j\mu^k\in S(\nu^j\mu^k,g)$ for all $j,\,k\in {\mathbb N}\cup\{0\}$.
\end{enumerate}
\end{remark}

\vskip 4mm


%
%

\begin{definition}
Let $b\in S(m,g)$ be 	a symbol and $u\in \S(\R^d)$, we set
\begin{align*}
  b(x,D,\tau) u (x) = \Op(b) u(x) := (2 \pi)^{-d}  \int_{\R^d} e^{i {x}\cdot{\xi}}
 \, b(x,\xi,\tau)\, \widehat{u}(\xi)  \ d \xi.
\end{align*}
 \end{definition}

It is known that $\Op(b) \,:\, \S(\R^d)\to \S(\R^d)$ is continuous and $\Op(b)$ can be uniquely extended to $\S'(\R^d)$  continuously.
The following two lemmas are consequences of Theorem 18.5.4 and 18.5.10 in \cite{Hormander83}.

\begin{lemma}
\label{lemma-remainder1}
Let $b\in S(m,g) $ where $m $ is an admissible weight and $g$ is defined by \eqref{eq:metric}.
Then there exists $c\in S(m,g) $ such that $\Op(b)^*=\Op(c)$ and $c(x,\xi)=
\overline{b(x,\xi)}+r(x,\xi)$ where  the remainder $r\in S(\lambda^{1\over2}\mu^{-1}m,g) $.
\end{lemma}

\begin{lemma}\label{lemma-remainder2}
Let $b\in S(m_1,g) $ and  $c\in S(m_2,g) $  where $m_j$ are  admissible weights for $j=1,2$ and
$g$ is defined by \eqref{eq:metric}.
Denote by
$
[Op(b),\, Op(c)] =
Op(b) \circ Op(c) - Op(c)\circ Op(b)
$ and   Poisson bracket
$\{b,\, c\}(x,\xi,\tau) = \sum\limits_{1\le j\le d}
(\partial_{\xi_j}b\, \partial_{x_j}c
-\partial_{x_j}b\, \partial_{\xi_j}c  )(x,\xi,\tau).$
Then,
\begin{enumerate}
\item[{\rm (i)}]
there exists $d\in S(m_1m_2,g)$ such that
$\Op(b)\Op(c)=\Op(d)$ and $d(x,\xi)=b(x,\xi)c(x,\xi)+r(x,\xi)$ where $r\in S(\lambda^{1\over2}\mu^{-1}m_1m_2,g) $.
\item[{\rm (ii)}]
for commutator $i[\Op(b),\Op(c)]=\Op(f)$, it holds that $f\in S(\lambda^{1\over2}\mu^{-1}m_1m_2,g) $
and $f(x,\xi)=\{ b,c\}(x,\xi)+r(x,\xi)$ where
$r\in S(\lambda \mu^{-2}m_1m_2,g) $.

\end{enumerate}
\end{lemma}

The operators in $S(\nu^j\mu^k,g)$ act on Sobolev spaces adapted to the class of symbol. Let
$b\in S(\nu^j\mu^k,g)$, where $\mu $ and
$g$ are defined by \eqref{eq:metric}. Then there exists $C>0$ such that
\begin{equation*}
\| \Op(b) u\|_{\R^d}\le C\|\nu^j \Op(\mu^k) u  \|_{\R^d},  \quad  \forall \;  u\in\S(\R^d).
\end{equation*}
By symbolic calculus, the above estimate is equivalent to $\Op(\mu^{-k}\nu^{-j})\Op(b)$ acts on $L^2(\R^d)$
since
the operators associated with symbol in $S(1,g) $  act on $L^2(\R^d)$.
In particular, if $b\in S(\nu^j \mu,g)$,    then for any $\lambda\ge \lambda_0 ,$ $ \tau \ge \max\{\,{\cal C}  \lambda , \; 1\,\} $
and  $u\in\S(\R^d), $ it holds
$$
\| \Op(b) u\|_{\R^d}\le C\tau \|\nu^j  u \|_{\R^d}  +C \| \nu^j  D u \|_{\R^d},
$$
where $C>0$ depends on positive constants $\lambda_0$ and ${\cal C}.$


%
%
\subsection{Commutator estimate}
In this subsection, we suppose that $\lambda =1$  since the symbol does not depend on $\lambda$. The  metric in  \eqref{eq:metric} becomes
\begin{equation}
\label{eq:metric1}
\tilde{g}=  dx^2+\mu^{-2}d\xi^2, \quad \;  \mbox{ where  }  \;  \mu  \;  \mbox{ is defined by \eqref{eq:metric}.}
\end{equation}

To get the commutator estimate, we shall use the following    G\aa rding inequality (\cite[Theorem 18.6.7]{Hormander83}).
\begin{lemma}\label{lemma-Garding}    Let $b\in S(\mu^{2k},\tilde{g})$ be real valued. $\mu $ and
$\tilde{g}$ are defined by \eqref{eq:metric1}. We assume
there exists $C>0$ such that
$b(x,\xi,\tau)\ge C\mu^{2k}$. Then there exist  $\widetilde{C}>0$ and $\tau_0>0$ such that
\begin{equation}\label{g-inequality}
\Re(\Op(b) w\,|\,w)_{\R^d} \ge \widetilde{C}\| \Op(\mu^k)w\|_{\R^d}^2 , \qquad \forall \; w\in\S(\R^d)\text{ and }\tau\ge \tau_0.
\end{equation}
\end{lemma}

\begin{definition}
\label{def: subelliptic}
Let $V$ be a bounded open set in $\R^d$. We say that the weight
 function $\varphi \in \Cinf(\R^d;\R)$ satisfies the  {\bf
   sub-ellipticity} condition in $\overline{V}$ if $\,|\,\nabla \varphi\,|\,>0$ in $\overline{V}$
 and  there exists constant $C>0$,
  \begin{equation}
    \label{eq: sub-ellipticity}
  \mathbbm{p}_\varphi(x,\xi,\tau)=0,\;\; \forall\, (x,\xi) \in \overline{V} \times\R^d,   \ \,\tau>0
   \quad
   \Rightarrow\quad  \{\mathbbm{q}_2,\mathbbm{q}_1\}(x,\xi,\tau) \geq C  (\,|\,\xi\,|^2\,+\tau^2)^{3\slash2},
 \end{equation}
 where $\mathbbm{p}_\varphi(x,\xi,\tau)=|\xi+i\tau\nabla\varphi (x)|^2=\mathbbm{q}_2(x,\xi,\tau)+i\mathbbm{q}_1(x,\xi,\tau)$ and $\mathbbm{q}_1,\;\mathbbm{q}_2$ are real valued.
 \end{definition}


\begin{lemma}\label{lemma-sub-ellip}
 (\cite{Hormander83}) Let $V$ be a bounded open set in $\mathbb{R}^d$
and $\psi \in \Con^\infty(\mathbb{R}^d;\mathbb{R})$  be such that $\,|\,\nabla\psi\,|\, > 0$ in $\overline{V}$. Then, for $\gamma>0$  sufficiently large, $\varphi = e^{\gamma \psi}$
fulfills the sub-ellipticity property in $V$.
\end{lemma}

\begin{lemma}\label{lemma-classical}
Assume that $\varphi $ satisfies the sub-ellipticity in Definition~\ref{def: subelliptic}.
For all  $w\in\Con_0^\infty(V)$,
there exist $C_1,C_2>0$ and $\tau_0>0$ such that the following inequality holds for all $\tau\ge\tau_0$,
\begin{equation}
\label{Classical-commutator-estimate}
\begin{array}{lcl}
C_1\tau^3\| w \|_{V}^2   + C_1\tau \| D w \|_{V}^2
 & \le&
   \Re \big( \Op (\{ |\xi|^2 -\tau^2|\nabla\varphi (x)|^2,\, 2\tau \xi\cdot\nabla\varphi (x) \}) w \, \,|\,\, w\big)_{V}
   \\ \noalign{\medskip} & &
  + \,C_2\tau^{-1}  \|\Op( |\xi|^2 -\tau^2|\nabla\varphi (x)|^2  )  w\|_{V}^2
\\ \noalign{\medskip}  && +\,C_2\tau ^{-1}   \|\Op( 2\tau \xi\cdot\nabla\varphi (x)  ) w \|_{V}^2 .
\end{array}
\end{equation}
\end{lemma}

\begin{proof}
First, by homogeneity  in $(\xi,\tau)$, compactness arguments and sub-ellipticity condition, we claim that there  exist  constants $C,\, \delta>0$   such that
\begin{equation}
\label{Carleman1}
C\big[\,|\,2\xi\cdot \nabla\varphi (x)\,|\,^2+\mu^{-2} \,(\,|\xi|^2-\tau^2|\nabla \varphi(x)|^2\,)\,^2
\big]+\big\{|\xi|^2 -\tau^2|\nabla\varphi(x)|^2,\; 2 \xi\cdot\nabla\varphi(x) \big\}\ge\delta \mu^2.
\end{equation}
\noindent
The proof of~\eqref{Carleman1} is classical.
In fact, set $${\cal K}=\{ (x,\xi,\tau) \in  {\mathbb R}^d\times {\mathbb R}^d\times {\mathbb R} \; :\; x\in\overline{V},\ \,|\,\xi\,|\,^2+\tau^2=1, \ \tau\ge0\}, $$
and for $ (x,\xi,\tau)\in {\cal K},\,\, \kappa>0,$
$$
    G(x,\xi,\tau,\kappa)=\kappa \big[\,|\,2\xi\cdot\nabla\varphi(x)\,|\,^2+\mu^{-2} \,(\,|\xi|^2-\tau^2|\nabla\varphi(x)|^2\,)\,^2
\big]+\{|\xi|^2 -\tau^2|\nabla\varphi(x)|^2, \;2 \xi\cdot\nabla\varphi(x) \}.$$

  \noindent
  If $\,|\,2\xi\cdot\nabla\varphi(x)\,|\,^2+\mu^{-2} \,(\,|\xi|^2-\tau^2|\nabla\varphi(x)|^2\,)\,^2=0$ for $(x,\xi,\tau)\in {\cal K}$,  it is clear that there exists a positive constant $\delta$
 such that \eqref{Carleman1} holds due to the fact that $\phi$ is sub-elliptic. When $\,|\,2\xi\cdot\nabla\varphi(x)\,|\,^2+\mu^{-2} \,(\,|\xi|^2-\tau^2|\nabla\varphi(x)|^2\,)\,^2>0$, there exists a positive constant $\kappa_{x,\xi,\tau}  $  such that
$G(x,\xi,\tau,\kappa)>0$ for every $\kappa\ge \kappa_{x,\xi,\tau}$  since $\{|\xi|^2 -\tau^2|\nabla\varphi(x)|^2, \;2 \xi\cdot\nabla\varphi(x) \}$
is bounded on ${\cal K}$. By continuity of $G(x,\xi,\tau,\kappa)$, there exists a   neighborhood of $(x,\xi,\tau)$, denoted by $V_{x,\xi,\tau}$,
such that $G(x,\xi,\tau,\kappa)>0$ for all  $(x,\xi,\tau)\in V_{x,\xi,\tau}$ and $\kappa\ge \kappa_{x,\xi,\tau}$. Since $ {\cal K}$ is  compact, there exist    finite sets  $V_j=V_{x_j,\xi_j,\tau_j}$ and corresponding constants $ \kappa_j= \kappa_{x_j,\xi_j,\tau_j} \;( j=1,\,2,\,\cdots ,n)$, such that ${\cal K}\subset \cup_{j=1}^n V_j $ and $G(x,\xi,\tau,\kappa)>0$ for all $(x,\xi,\tau)\in V_j$ and $\kappa>\kappa_j$.
  Let $\tilde\kappa=\max \{\kappa_j\;:\: j=1,2,\cdots, n\}$. It follows that $G(x,\xi,\tau,\tilde\kappa)>0$ for all
$(x,\xi,\tau)\in {\cal K}$ and $\kappa\ge \tilde{\kappa} $.
Finally, using the compactness of  ${\cal K}$ again,  we conclude that there exists $\delta>0$ such that $G(x,\xi,\tau,\tilde\kappa)\ge \delta$.
Thus,  ~\eqref{Carleman1} is reached since  $g$ is a homogeneous function of degree 2 with respect to variables $ (\xi,\tau )$.

By G\aa rding inequality \eqref{g-inequality}, there exists a constant $C>0$ such that, for $\tau\ge\tau_0$ with $\tau_0$ sufficiently large,
\begin{equation}
\label{Carleman2}
\begin{array}{ll}
C  \| \Op(\mu) w \|_{V}^2\le & \Re \big(\Op\big( \,\,|\,2\xi\cdot \nabla\varphi (x)\,|\,^2+\mu^{-2} \,(\,|\xi|^2-\tau^2|\nabla \varphi(x)|^2\,)\,^2
\\ \noalign{\medskip} & \displaystyle
 \quad\;\;   +\big\{|\xi|^2 -\tau^2|\nabla\varphi(x)|^2,\; 2 \xi\cdot\nabla\varphi(x) \big\} \big)w\; \big|\, w\big)_{V}.
\end{array}
\end{equation}

\noindent
 Now we are going to   estimate the terms
$ \tau\Op(\,|\,2\xi\cdot \nabla\varphi (x)\,|\,^2)$ and $\tau\Op(\mu^{-2} \,(\,|\xi|^2-\tau^2|\nabla \varphi(x)|^2\,)\,^2)$.
Firstly, it follows from  Lemma \ref{lemma-remainder2} that
\begin{equation}
\label{Carleman3}
\tau^{-1}\Op(\,|\,2\tau\xi\cdot \nabla\varphi (x)\,|\,^2) = \tau^{-1}\Op(2\tau\xi\cdot \nabla\varphi (x))^*\Op(2\tau\xi\cdot \nabla\varphi (x))+\tau\Op(r_1),
\end{equation}
where $r_1\in S(\mu,\tilde{g})$ and $\tilde{g} $ is defined by \eqref{eq:metric1}.
Therefore, for any $\varepsilon>0$, there exists a positive constant $C_\varepsilon$ such that
\begin{equation}\label{Carleman61}\begin{array}{ll}&
\big|\big( \tau^{-1}\Op(\,|\,2\tau\xi\cdot \nabla\varphi (x)\,|\,^2)  w\;|\;w\big)_{V} \big|
\\ \noalign{\medskip} \le &  \tau^{-1}\| \Op(  2\tau\xi\cdot \nabla\varphi (x))w\|_{V}^2+ \tau \,|\,( \Op(r_1)w\,|\,w)_{V}\,|\,
\\ \noalign{\medskip} \le &   \tau^{-1}\|  \Op(  2\tau\xi\cdot \nabla\varphi (x))w\|_{V}^2+
 \varepsilon \, \tau \| \Op(\mu )w \|_{V}^2 +C_\varepsilon\tau\| w \|_{V}^2 .
\end{array}\end{equation}

\noindent
Substituting \eqref{Carleman61} into \eqref{Carleman2}
and choosing  $\varepsilon $ small enough, we have  
\begin{equation}
\label{Carleman7}\begin{array}{ll}
& C\tau\| \Op(\mu) w \|_{V}^2
\\ \noalign{\medskip}
 \le  &   \Re \big(\Op\big( \tau\mu^{-2} \,(\,|\xi|^2-\tau^2|\nabla \varphi(x)|^2\,)\,^2
  +\tau\big\{|\xi|^2 -\tau^2|\nabla\varphi(x)|^2,\; 2 \xi\cdot\nabla\varphi(x) \big\} \big)w\; \big|\, w\big)_{V}
  \\ \noalign{\medskip}  &
  +\tau^{-1}\|  \Op(  2\tau\xi\cdot \nabla\varphi (x))w\|_{V}^2+C_\varepsilon\tau\| w \|_{V}^2 .
\end{array}\end{equation}

\noindent
Secondly,  by symbolic calculus, we have that
\begin{equation}\label{Carleman71}
\tau\Op(\mu^{-2}\,(\,|\xi|^2-\tau^2|\nabla \varphi(x)|^2\,)\,^2)   =  \tau \Op(r_0)\Op(|\xi|^2-\tau^2|\nabla \varphi(x)|^2)  +\tau\Op(r_2),
\end{equation}
where $r_0(x,\xi) =\mu^{-2}(|\xi|^2-\tau^2|\nabla \varphi(x)|^2)\in S(1,\tilde{g})$ and $r_2\in S(\mu,\tilde{g})$.
Therefore, for all $\varepsilon>0$, there exists  $C_\varepsilon>0$ such that
\begin{equation}\label{Carleman8}
\begin{array}{ll}&
\,|\, (\tau \Op(r_0)\Op(|\xi|^2-\tau^2|\nabla \varphi(x)|^2) w  \,|\, w )_{V} \,|\,
\\ \noalign{\medskip} \le&
  C_\varepsilon\tau^{-1}\| \Op(|\xi|^2-\tau^2|\nabla \varphi(x)|^2) w  \|_{V}^2+  \varepsilon \tau ^3
\| w \|_{V}^2.
\end{array}\end{equation}
We choose $\varepsilon$ small enough and combine \eqref{Carleman71}-\eqref{Carleman8} with \eqref{Carleman7} to get
\begin{equation}
\label{Carleman9}\begin{array}{ll}
 C\tau\| \Op(\mu) w \|_{V}^2
  \le  &   \Re\big( \tau\Op\big\{|\xi|^2 -\tau^2|\nabla\varphi(x)|^2,\; 2 \xi\cdot\nabla\varphi(x) \big\} w\; \big|\, w\big)_{V}
  \\ \noalign{\medskip} & \displaystyle +C_\varepsilon\tau^{-1}\| \Op(|\xi|^2-\tau^2|\nabla \varphi(x)|^2) w  \|_{V}^2
  \\ \noalign{\medskip} &
  + \tau^{-1}\| \Op(2\tau\xi\cdot\nabla\varphi(x)) w\|_{V}^2+C_\varepsilon(\tau+ \varepsilon \tau ^3)\| w \|_{V}^2 .
\end{array}\end{equation}

\noindent
Finally, it is clear that there exist positive constant $C$ such that
\begin{equation}
\label{Carleman4}
 \tau^3\| w \|_{V}^2   +  \tau \| D w \|_{V}^2   \le C\tau \| \Op(\mu) w \|_{V}^2.
\end{equation}
Thus, we obtain \eqref{Classical-commutator-estimate} by using \eqref{Carleman9}-\eqref{Carleman4}, choosing $\varepsilon$ small enough and letting $\tau>\tau_0$ large enough.
\end{proof}


%
%

\section{Carleman Estimate}
\setcounter{equation}{0} \setcounter{theorem}{0}
In this section, we shall prove several Carleman inequalities.
%
  Define the operator  $$P(x,D,\lambda)=D^2+i\lambda D a(x)D-\lambda^2.$$ Let the weight function $\varphi \in \Cinf(\R^d;\R)$.
The associated conjugate operator of $P(x,D,\lambda)$ is  $P_\varphi(x,D,\lambda)= e^{\tau \varphi}P(x,D,\lambda) e^{-\tau \varphi}$. Then,
$$
P_\varphi=(D+i\tau\nabla\varphi (x))^2+i\lambda (D+i\tau\nabla\varphi (x)) a(x)(D+i\tau\nabla\varphi (x))-\lambda^2.
$$
By setting $Q_2=\frac12(P_\varphi + P_\varphi ^*)$  and  $Q_1=\frac1{2i}(P_\varphi - P_\varphi ^*)$, we have $P_\varphi=Q_2+iQ_1$. 
We denote by $p(x, \xi, \lambda), \;p_\varphi(x, \xi, \lambda)  $  the associated symbol of $P(x,D, \lambda),$ $P_\varphi(x,D, \lambda),$ respectively.

Let the metric $g$ and weight $\nu$ be defined by \eqref{eq:metric} and \eqref{eq:weight}.
Due to the results Remark \ref{remark-class1}, we know that $D+i\tau\nabla\varphi (x)$ is an operator with symbol in $S(\mu,g) $ class, $\lambda a$ is in $S(\nu,g)$, and $(1+i\lambda a(x)) |\xi+i\tau\nabla\varphi (x)|^2$,
the principal symbol of $P_\varphi $  belongs to $S(\nu\mu^2,g)$. Furthermore, from Lemma \ref{lemma-remainder1} and \ref{lemma-remainder2}, one has that
\begin{align}
&P_\varphi=\Op \big((1+i\lambda a(x)) |\xi+i\tau\nabla\varphi (x)|^2\big)-\lambda^2 +R_3,  \notag \\
&Q_2=\Op(q_2)-\lambda^2 +R_2 , \notag \\
& Q_1=\Op(q_1)+R_1, \label{eq: symbol  Q_j and remainders}
\end{align}
where $q_2=|\xi|^2-\tau^2|\nabla\varphi (x)|^2-2\lambda \tau a(x)\xi\cdot\nabla\varphi (x) $,~ $q_1= 2\tau \xi\cdot\nabla\varphi (x) +\lambda a(x)
(|\xi|^2-\tau^2|\nabla\varphi (x)|^2)$  belong  to $S(\nu\mu^2,g)$ and the symbols of $R_j$ is in $S(\lambda^{1\over2}\nu\mu,g)$ for $j=1,2,3$.
It is clear that
\begin{equation}\label{eq: Carleman commutator and square}
\|  P_\varphi v \|_V^2=\| Q_2 v\|_V^2 +\| Q_1 v\|_V^2+ 2\Re (Q_2 v\,|\, iQ_1 v)_V,
 \quad  \forall\; v
 \in \Cinf(V) .
\end{equation}

In what follows, several Carleman estimates are introduced.
First, we give an estimation on the subdomain which is far away from the boundary $\Gamma.$


 \begin{theorem}
\label{th-Carleman1}
Suppose $\varphi$ satisfies   sub-ellipticity condition in $V\subset \Omega$.
Then, there exist positive constants $C,  \;\tilde K$ and $\lambda_0$, such that for every  $v\in\Con_0^\infty(V)$,   it holds
\begin{equation}\label{Carleman estimate L2}
   \tau^3\| \nu(x) v\|_{V}^2   +   \tau   \|\nu(x)  D v \|_{V}^2  +  \| Q_1 v \|_{V}^2 +   \| Q_2 v \|_{V}^2
   \lesssim  \| P_\varphi v\|_{V}^2  ,
\end{equation}
where  $\lambda\ge \lambda_0$ and
$\tau\ge  \max\{\tilde K|\lambda|^{5\over4},\,1\}$.
\end{theorem}
\begin{proof}
 Let $w =\nu(x) v$ in  \eqref{Classical-commutator-estimate}. We obtain
\begin{equation} \label{eq: commutator estimate}\begin{array}{ll}
&C_1\tau^3\| \nu(x) v\|_{V}^2   + C_1\tau \| D( \nu(x) v) \|_{V}^2
\\ \noalign{\medskip} \le & \Re\big( \Op (\big\{ |\xi|^2 -\tau^2|\nabla\varphi(x)|^2, 2\tau \xi\cdot\nabla\varphi (x)  \big\}) \nu(x) v  \,|\, \nu(x) v\big)_{V}  \\ \noalign{\medskip}
&  +\, C_2\tau^{-1}  \|\Op( |\xi|^2 -\tau^2|\nabla\varphi(x)|^2  ) \nu(x) v\|_{V}^2   \\\noalign{\medskip}
&   +\, C_2\tau ^{-1}   \|\Op( 2\tau \xi\cdot\nabla\varphi (x)  )\nu(x) v \|_{V}^2 , \;\; C_1,\,C_2>0.
\end{array}
\end{equation}

\noindent
 Since the symbol of $[D,\nu(x)]$ is in $S(\lambda^{1\over2}\nu,g)$, we have
\begin{equation*}
 \|\nu(x)  D v \|_{V}^2   \le \| D (\nu(x) v) \|_{V}^2 +   \| [D,\nu(x)] v \|_{V}^2
 \le  \| D (\nu(x) v) \|_{V}^2 + C \lambda  \| \nu(x) v \|_{V}^2.
\end{equation*}
Consequently, for $\tau\ge {\cal C}\lambda $ with ${\cal C}>0$ sufficiently large, it holds
\begin{equation} \label{eq: estimation by below}
\tau^3\| \nu(x) v\|_{V}^2   + \tau   \|\nu(x)  D v \|_{V}^2 \lesssim   \tau^3\| \nu(x) v\|_{V}^2   + \tau \| D (\nu(x) v )\|_{V}^2 .
\end{equation}
It follows from \eqref{eq: commutator estimate} and \eqref{eq: estimation by below} that
\begin{equation} \label{eq: commutator estimate0}
\begin{array}{ll} &
C_1'\big(\tau^3\| \nu(x) v\|_{V}^2   +  \tau \|   \nu(x) D v  \|_{V}^2\big)
 \\ \noalign{\medskip}   \le & \Re \big( \Op (\big\{ |\xi|^2 -\tau^2\varphi(x)^2, 2\tau \xi\cdot\nabla\varphi (x)  \big\}) \nu(x) v  \,|\, \nu(x) v\big)_{V}  \\ \noalign{\medskip}
&  + C_2\tau^{-1}  \|\Op( |\xi|^2 -\tau^2|\nabla\varphi(x)|^2  ) \nu(x) v\|_{V}^2  \\ \noalign{\medskip}
&   + C_2\tau ^{-1}   \|\Op( 2\tau \xi\cdot\nabla\varphi (x)  )\nu(x) v \|_{V}^2 , \;\; C_1',\,C_2>0.
\end{array}
\end{equation}

Now, we estimate the first term on the right side hand of \eqref{eq: commutator estimate0}. Note that
\begin{equation}\label{eq: real part and commutator}
2\Re (Q_2 v\,|\, iQ_1 v)_{V}= (Q_2 v\,|\, iQ_1 v)_{V}+ (  iQ_1  v\,|\, Q_2v)_{V}=( i[Q_2, Q_1 ] v \,|\, v )_{V},
\end{equation}
where   the principal symbol of $ i[Q_2, Q_1 ] $ is
$\{q_2, q_1\} $. Due to Lemma \ref{lemma-remainder2},   we obtain that
\begin{equation}
\label{eq:qq}
i[Q_2, Q_1 ] = \Op(\{  q_2, q_1  \}   ) +R_4,
\end{equation}
where $\{  q_2, q_1  \}  \in S(\lambda^{\frac12}\nu^2 \mu^3, g)$ and the symbol of $R_4$ is in $S(\lambda\nu^2 \mu^2, g)$.
A direct computation gives that
\begin{align*}
\big\{  q_2, q_1  \big\}&= (1+a^2(x) \lambda ^2) \big\{ |\xi|^2-\tau^2|\nabla\varphi(x)|^2  , 2 \tau \xi\cdot\nabla\varphi (x)  \big\} \\
&\quad +\big(  \big\{|\xi|^2-\tau^2|\nabla\varphi(x)|^2   , \lambda a(x)  \big\} -\lambda a(x) \big\{ 2\tau\xi.\nabla\varphi (x),\lambda a(x)\big\}
\big) (|\xi|^2-\tau^2|\nabla\varphi(x)|^2)  \\
&\quad - 2   \tau \xi\cdot\nabla\varphi (x)\big(  \lambda  a(x) \big\{\lambda a(x),|\xi|^2-\tau^2|\nabla\varphi(x)|^2\big\} +  \big\{\lambda a(x)  , 2\tau \xi\cdot\nabla\varphi (x)  \big\}
\big) .
\end{align*}

\noindent
From the definition of $q_1$ and $q_2$, we have
\begin{align}
&q_2(x,\xi)+\lambda a(x) q_1(x,\xi)=  ( 1+\lambda^2a(x)^2 ) (|\xi|^2-\tau^2|\nabla\varphi(x)|^2 ),\notag\\
&q_1(x,\xi)- \lambda a(x) q_2(x,\xi)  =     2\tau \xi\cdot\nabla\varphi (x)   ( 1+\lambda^2a^2 (x)).  \label{eq: q-j versus classical terms}
\end{align}

\noindent
Consequently,
\begin{equation*}
\begin{array}{ll}&
\big\{  q_2, q_1  \big\}
\\ \noalign{\medskip} =& \displaystyle
\nu^2(x) \big\{ |\xi|^2-\tau^2|\nabla\varphi(x)|^2  ,\, 2 \tau \xi\cdot\nabla\varphi (x)  \big\}
\\ \noalign{\medskip} & \displaystyle
   + \nu^{-2}(x)  (q_2(x,\xi)+\lambda a(x) q_1(x,\xi))            \big(  \big\{|\xi|^2-\tau^2|\nabla\varphi(x)|^2   ,
 \lambda a(x)  \big\} -\lambda a(x) \big\{ 2\tau\xi\cdot\nabla\varphi (x),\lambda a(x)\big\}
\big)
 \\ \noalign{\medskip} & \displaystyle
   - \nu^{-2}(x)   (q_1(x,\xi)- \lambda a(x) q_2(x,\xi) )
\big( \lambda  a(x) \big\{\lambda a(x),|\xi|^2-\tau^2|\nabla\varphi(x)|^2\big\} +  \big\{\lambda a(x)  , 2\tau \xi\cdot\nabla\varphi (x)  \big\}
\big)
.
\end{array}
\end{equation*}


\noindent
Then, it follows from  \eqref{eq:qq} and the above equation that
\begin{align*}
 i[Q_2, Q_1 ] &=  \nu(x) \Op \big( \big\{ |\xi|^2-\tau^2|\nabla\varphi(x)|^2  ,\, 2 \tau \xi\cdot\nabla\varphi (x)  \big\} \big) \nu(x) \notag \\
 &\quad +B_1\nu^{-1}(x)  \Op \big(q_2(x,\xi)+\lambda a(x) q_1(x,\xi) \big) \notag  \\
  &\quad -  B_2\nu^{-1}(x)  \Op \big  (q_1(x,\xi)- \lambda a(x) q_2(x,\xi) \big) + R_5,
\end{align*}
where
$$
\begin{array}{l}
B_1 = \Op\big(\nu^{-1}(x)
\big(  \big\{|\xi|^2-\tau^2|\nabla\varphi(x)|^2   , \lambda a(x)  \big\} -\lambda a(x) \big\{ 2\tau\xi\cdot\nabla\varphi (x),\lambda a(x)\big\}\big)\big),
 \\ \noalign{\medskip}   \displaystyle
 B_2 = \Op\big(\nu^{-1}(x)
    \big(  \lambda  a(x) \big\{\lambda a(x),|\xi|^2-\tau^2|\nabla\varphi(x)|^2\big\} +  \big\{\lambda a(x)  , 2\tau \xi\cdot\nabla\varphi (x)  \big\}
\big)\big),
\end{array}
$$
  symbols of $B_1$ and $B_2$ belong to $ S( \lambda ^{1\over2}\nu\mu ,g)$ and
 the symbol of $R_5$ is in $S(\lambda\nu^2\mu^2,g)$.
 Combining this with \eqref{eq: symbol  Q_j and remainders} yields
\begin{align}
 i[Q_2, Q_1 ] &=  \nu(x) \Op \big( \big\{ |\xi|^2-\tau^2|\nabla\varphi(x)|^2  , 2 \tau \xi\cdot\nabla\varphi (x)  \big\} \big) \nu(x) \notag \\
 &\quad +    B_1\nu^{-1}(x)    \big[Q_2(x,\xi)+\lambda a(x) Q_1(x,\xi)+\lambda^2\big]  \notag  \\
  &\quad -   B_2  \nu^{-1}(x)    \big[Q_1(x,\xi)- \lambda a(x) Q_2(x,\xi)-\lambda^3a(x) \big] + R_5,  \label{eq: commutator computation}
\end{align}


\noindent
 We refer to Section 2 where the rules on symbolic
 calculus are given and precise.
 Therefore, by the continuity of pseudo-differential operator,
we have that for   $j=1,2$, $k=0,1$ and $\ell=1,2$,
\begin{equation}\label{eq: estimate B_j}
\begin{array}{ll}&
\,|\,(B_j \nu(x)^{-1}  (\lambda a(x)) ^k Q_\ell v\,|\, v)_{V}\,|\,
\\ \noalign{\medskip}  = &\,|\,( \nu(x)^{-1}   (\lambda a(x)) ^kQ_\ell v\,|\,  B_j^*v)_V\,|\,\lesssim\, \| Q_\ell v \|_{V} \|  B_j^*v  \|_{V}
  \\ \noalign{\medskip} \le & \displaystyle {1\over10}  \| Q_\ell v \|_{V}^2 +C \lambda \tau ^2\|\nu(x) v\|_{V} ^2+C \lambda \|\nu(x)  Dv\|_{V} ^2
, \quad  C>0,
\end{array}
\end{equation}
\begin{align}
\,|\,(  B_j\nu^{-1}(x)\lambda^2(\lambda a(x))^k v\,|\, v  )_{V}   \,|\,&\,\lesssim\,\lambda^{5\over2}\big(\tau \|\nu(x) v \|_{V} +\|\nu(x) Dv \|_{V}  \big) \|\nu(x) v \|_{V} \notag\\ \noalign{\medskip}
&\,\lesssim\,\lambda^{5\over2}\tau \|\nu(x) v \|_{V} ^2+  \lambda^{5\over2}\tau^{-1} \|\nu(x) Dv \|_{V}^2,
 \label{eq: estimate non principally terms}
\end{align}
and
\begin{equation}\label{eq: estimate remainders}
\,|\,(R_5 v\,|\, v)_{V}\,|\,\lesssim \lambda \tau^2 \|\nu(x) v \|_{V}^2 + \lambda  \|\nu(x) Dv \|_{V}^2.
\end{equation}


\noindent
Due to~\eqref{eq: commutator computation}--\eqref{eq: estimate remainders},   there exists a positive constant $C$ such that
\begin{equation}\label{eq: estimate commutator}
\begin{array}{ll}
&\big( \Op (\big\{ |\xi|^2 -\tau^2|\nabla\varphi(x)|^2, 2\tau \xi\cdot\nabla\varphi (x)  \big\}) \nu(x) v  \,|\, \nu(x) v\big)_{V}\\  \noalign{\medskip}
\le& \displaystyle (i[Q_2,Q_1] v \,|\,v)_{V}
+ {1\over4}\sum\limits_{\ell=1,2}  \| Q_\ell v \|_{V}^2   +C\big(\lambda \tau ^2\|\nu(x) v\|_{V} ^2
\\\noalign{\medskip}
& + \lambda \|\nu(x)  Dv\|_{V} ^2    + \lambda^{5\over2}\tau \|\nu(x) v \|_{V} ^2+ \lambda^{5\over2}\tau^{-1} \|\nu(x) Dv \|_{V}^2\big).
\end{array}
\end{equation}

Next, we are going to estimate the last two terms on the right side hand of \eqref{eq: commutator estimate0}.
It follows from ~\eqref{eq: symbol  Q_j and remainders} and \eqref{eq: q-j versus classical terms}
that
\begin{equation}
\label{eq:q1}
\begin{array}{ll}&
\|\Op( |\xi|^2 -\tau^2|\nabla\varphi(x)|^2  ) \nu(x) v\|_{V}^2 +  \|\Op( 2\tau \xi\cdot\nabla\varphi (x) )\nu(x) v \|_{V}^2
\\ \noalign{\medskip}=
& \|\Op(\nu(x)^{-2}( q_2+\lambda a(x) q_1 ) \nu(x) v\|_{V}^2 +
\|\Op(\nu(x)^{-2}( q_1-\lambda a(x) q_2 ) \nu(x)  v \|_{V}^2
\\ \noalign{\medskip}=
& \| \nu(x)^{-1} \big[  Q_2+\lambda^2 - R_2+\lambda a(x) (Q_1-R_1 )\big] v\|_{V}^2
\\ \noalign{\medskip}
&+
\| \nu(x)^{-1}\big[ Q_1-R_1-\lambda a(x)(Q_2+\lambda^2 - R_2) \big] v \|_{V}^2.
\end{array}
\end{equation}

\noindent
Combining this with the fact that the symbols of $ R_j$ are in $S(\lambda^{1\over2}\nu \mu,g)$ for $j=1,2$, we have
\begin{equation} \label{eq: estimate by Q_j}
\begin{array}{ll}
&
\|\Op( |\xi|^2 -\tau^2|\nabla\varphi(x)|^2  ) \nu(x) v\|_{V}^2 +  \|\Op( 2\tau \xi\cdot\nabla\varphi (x) )\nu(x) v \|_{V}^2
\\ \noalign{\medskip}
\lesssim &\displaystyle
   \sum\limits_{\ell=1,2 }(\| Q_\ell  v  \|_{V}^2  +  \| R_\ell v  \|_{V} ^2 )   + \lambda^4 \|  v  \|_{V} ^2
 \\ \noalign{\medskip}
\lesssim &
 \sum\limits_{\ell=1,2 }\| Q_\ell  v  \|_{V}^2 +  \lambda \tau^2 \| \nu(x) v  \|_{V} ^2 +   \lambda\|\nu(x) D v  \|_{V} ^2   + \lambda^4 \|  v  \|_{V} ^2 .
 \end{array}\end{equation}


\noindent
Consequently, for $ \tau\ge {\cal C}\lambda$ with ${\cal C}$ large enough,   it holds
\begin{equation}\label{eq: estimate square}
\begin{array}{ll}&
   \tau^{-1}  \|\Op( |\xi|^2 -\tau^2|\nabla\varphi(x)|^2  ) \nu(x) v\|_{V}^2  +\tau ^{-1}   \|\Op( 2\tau \xi\cdot\nabla\varphi (x)  )\nu(x) v \|_{V}^2
\\ \noalign{\medskip}   \lesssim &\displaystyle
   \tau^{-1}\sum_{\ell=1,2 }\limits\| Q_\ell  v  \|_{V}^2 + \lambda \tau \|\nu(x)  v  \|_{V} ^2 +  \|\nu(x) D v  \|_{V} ^2   + \lambda^3 \|  v  \|_{V} ^2 .
\end{array}
\end{equation}


\noindent
Finally, by \eqref{eq: commutator estimate0}, \eqref{eq: estimate commutator} and \eqref{eq: estimate square}, one can choose $\tau\ge \max\{\tilde K|\lambda|^{5\over4},\, 1\}$
with $ \tilde K$ sufficiently
large  such that for some $C>0$,
\begin{equation}\label{eq: estimate combine}
\begin{array}{ll}
C \big( \tau^3\| \nu(x) v\|_{V}^2   +  \tau   \|\nu(x)  D v \|_{V}^2\big)
\le & \displaystyle  (i[Q_2,Q_1] v\,|\,v)_{V} +  {1\over2}\sum\limits_{\ell=1,2}  \| Q_\ell v \|_{V}^2
\\ \noalign{\medskip} &
+ \, \varepsilon  \tau^3\| \nu(x) v\|_{V}^2   +\varepsilon \tau   \|\nu(x)  D v \|_{V}^2,
\end{array}
\end{equation}
where $\varepsilon>0$ is arbitrary. Choosing $\varepsilon$ small with respect to
$C$, using \eqref{eq: Carleman commutator and square}, \eqref{eq: real part and commutator} and \eqref{eq: estimate combine},
we obtain
$$
\begin{array}{ll}&\displaystyle
C_1\big( \tau^3\| \nu(x) v\|_{V}^2   +   \tau   \|\nu(x)  D v \|_{V}^2\big) + {1\over2}\big(\| Q_1 v \|_{V}^2 +   \| Q_2 v \|_{V}^2\big)
\\ \noalign{\medskip}
  \le  & \displaystyle  (i[Q_2,Q_1] v,v)_V + \| Q_1 v \|_{V}^2 + \| Q_2 v \|_{V}^2= \| P_\varphi v\|_{V}^2  ,
\end{array}
$$
which implies~\eqref{Carleman estimate L2}.
\end{proof}
%
%

\vskip 4mm

\begin{remark}
The estimates in \eqref{eq: estimate non principally terms} impose the assumption $\tau\ge \tilde K |\lambda|^{5\over4}$. The other remainder terms only impose the
condition $\tau\ge {\cal C}\lambda$.  This condition is related with the principal normal condition. Indeed for a complex operator, with symbol $p_1+ip_2$    
where $p_1,p_2$ are both real valued, the Carleman estimate is only true if $\{ p_1,p_2 \}=0$ on $p_1=p_2=0$. Here the symbol of operator before
conjugation by weight is $|\xi|^2-\lambda^2+i\lambda a(x)|\xi|^2$, and the Poisson bracket is
$\{ |\xi|^2-\lambda^2, \lambda a(x)|\xi|^2 \}=2\lambda (\xi\cdot \nabla a(x))|\xi|^2$. We can  estimate this term, uniformly in a neighborhood
of $a(x)=0$, by $C\lambda a^{1\over2}(x)\,|\,\xi\,|\,^3$. This explanation does not justify the power $|\lambda|^{5\over4}$ found at the end of computations but
shows the difficulties.
\end{remark}

\begin{theorem}
\label{th-Carleman}
Suppose $\varphi$ satisfies   sub-ellipticity condition in $V\subset \Omega$.
Then, there exist positive constants $  \;\tilde K$ and $\lambda_0$, such that for every  $u\in\Con_0^\infty(V)$,   it holds
\begin{equation}\label{eq: Carleman estimate}
\tau^3\|( 1+\lambda^2a(x)^2 )^{1\over2} e^{\tau \varphi} u \|_{V}^2+\tau\|  ( 1+\lambda^2a(x)^2 )^{1\over2} e^{\tau \varphi} Du \|_{V}^2\lesssim
\| e^{\tau \varphi}Pu \|_{V}^2 ,
\end{equation}
where  $\lambda\ge \lambda_0$ and
$\tau\ge  \max\{\tilde K|\lambda|^{5\over4},\,1\}$.
\end{theorem}

\begin{proof}
 Set $v=e^{\tau\varphi}u$. From Theorem \ref{th-Carleman1}, it suffices to prove   that  \eqref{eq: Carleman estimate} is equivalent to
\begin{equation}\label{eq: conjugate Carleman estimate}
 \tau^3 \| \nu(x)  v \|_{V}^2+  \tau \| \nu(x) D  v \|_{V}^2\lesssim \|  P_\varphi v \|_{V}^2 .
\end{equation}
 First, assume \eqref{eq: conjugate Carleman estimate} holds.   Then,
  $Dv=e^{\tau\varphi}(Du-i\tau \nabla \varphi  \, u)$  and  $e^{\tau\varphi}Du= Dv +i\tau \nabla\varphi \,v$.
  Then there exist positive constants $c_1,\; c_2$ such that
  \begin{equation}
  \label{eq:dc1c2}
  \begin{array}{lcl}
  c_1\,\big(\|\, \nu(x)  Dv \|_{V}+\tau \| \nu(x) v\|_{V}\big) & \le& \| \nu(x) e^{\tau \varphi} Du \|_{V} +\tau \| \nu(x) e^{\tau \varphi}  u \|_{V}
  \\ \noalign{\medskip} &
  \le & c_2\,\big(\| \nu(x)  Dv \|_{V}+\tau \| \nu(x) v\|_{V}\big) .
  \end{array}
\end{equation}
  Combining this with \eqref{eq: conjugate Carleman estimate},
   we conclude that
 \begin{align*}
 \tau^3\| \nu(x) e^{\tau \varphi} u \|_{V} ^2+\tau\| \nu(x) e^{\tau \varphi} Du \|_{V}^2 \lesssim  ( \tau \| \nu(x) D  v \|_{V}^2+  \tau^3 \| \nu(x)  v \|_{V}^2)
 \lesssim
  \|e^{\tau\varphi}  P u\|_{V}^2.
 \end{align*}

 \noindent
 On the other hand, ~\eqref{eq: Carleman estimate} implies that
 \begin{align*}
 \tau^3\| \nu(x) e^{\tau \varphi} u \|_{V} ^2+\tau\| \nu(x) e^{\tau \varphi} Du \|_{V}^2
 \lesssim\|  P_\varphi v \|_{V}^2.
 \end{align*}
 Then, we proved \eqref{eq: conjugate Carleman estimate} from the above estimate and \eqref{eq:dc1c2}.
 \end{proof}

\vskip 4mm

Since there is higher order term $\Div(a(x)\nabla y_t)$ in system \eqref{system},  it is necessary to deal with the term
$\Div(a(x)\nabla f)$ for $f\in H^1(\Omega)$ when proving the resolvent estimate.
The following result is analogue to the work by   Imanuvilov and  Puel  (\cite{puel}).

\begin{theorem}
	\label{th: Carleman H-1}
Suppose $\varphi$ satisfies   sub-ellipticity condition on $V\subset \Omega$.
Then, there exist $C, \,\tilde K,\,\lambda_0>0$,   such that for all  $u\in\Con_0^\infty(V)$ satisfying
\begin{equation}
\label{eq:Pf}
P(x,D,\lambda)u=g_0+\sum_{j=1} ^d  \partial_{x_j} g_j, \; \;  \mbox{where   }  \; g_j \in L^2(V),\; j=0,1,\cdots, d,
\end{equation}
it holds
\begin{equation}\label{eq: Carleman estimate H-1}
\tau\|( 1+\lambda^2a(x)^2 )^{1\over2} e^{\tau \varphi} u \|_{V } ^2+\tau^{-1} \|  ( 1+\lambda^2a(x)^2 )^{1\over2} e^{\tau \varphi} Du \|_{V}^2\le
C\sum_{j=0}^d  \| e^{\tau \varphi}g_j \|_{V}^2,
\end{equation}
where  $\lambda\ge \lambda_0$ and
$\tau\ge  \max\{\tilde K|\lambda|^{5\over4},\,1\}$.
 \end{theorem}
%

\begin{proof}
First, from \eqref{eq: symbol  Q_j and remainders}, we have $D^2= Q_2+S_2$ and $ \lambda a(x) D^2= Q_1+S_1$
where $S_1$ and $S_2$ have  symbols in $S(\tau \nu\mu,g)$  if $\tau\gtrsim  \lambda $. It follows that for any $  v
 \in \Cinf_0(V),$
\begin{equation}
\label{eq: estimate D2}
 \begin{array}{l}
 \| D^2 v\|^2_V\lesssim \| Q_2 v  \|^2_V  +\tau^2 ( \tau^2\| \nu(x) v\|^2_V   +  \|\nu(x)  D v \|^2_V),
 \\ \noalign{\medskip}
 \| \lambda aD^2  v \|^2_V\lesssim  \| Q_1v  \|^2_V  +\tau^2 ( \tau^2\| \nu(x) v\|^2_V   +  \|\nu(x)  D v \|^2_V).
\end{array}
 \end{equation}
Using \eqref{Carleman estimate L2}, \eqref{eq: estimate D2} and the fact that $\| \nu(x) D^2 v\|^2_V \le 2( \| D^2 v\|^2_V+  \| \lambda a(x) D^2  v \|^2_V)$,
  we obtain
\begin{equation}
	\label{eq: second derivative estimate}
\begin{array}{ll} & \displaystyle
\tau^3\| \nu(x) v\|^2_V   +  \tau   \|\nu(x)  D v \|^2_V+
\tau^{-1} \|\nu(x) D^2 v\|^2_V
 \lesssim   \| P_\varphi v\|^2_V  .
\end{array}
\end{equation}
%
%

Let $\widetilde{u} $ and $\chi$ be in $ \Con_0^\infty(\Omega) $ such that $\chi=1$ on a \nhd  of $\supp \widetilde{u}$.
Similarly to~\eqref{eq:dc1c2},
we obtain
\begin{equation}
\label{eq: estimation by mu-1}
\tau \| \nu(x) \widetilde{u}\|^2_V + \tau ^{-1}\| \nu(x) D \widetilde{u}\|^2_V
 \lesssim \tau \| \nu(x) \widetilde{u}\|^2_V + \tau ^{-1}\|  D ( \nu(x) \widetilde{u})\|^2_V  \notag.
\end{equation}
Then, combining this with Fourier transform  and the  following inequality
\begin{align*}
  \tau+\frac{\,|\,\xi\,|\,^2}{\tau}   \lesssim \frac{\tau^3}{\tau^2+\,|\,\xi\,|\,^2}+\frac{\,|\,\xi\,|\,^4}{\tau( \tau^2+\,|\,\xi\,|\,^2 )} ,
\end{align*}
we conclude   that
\begin{equation}
\label{eq: estimation by mu-1}
\tau \| \nu(x) \widetilde{u}\|^2_V + \tau ^{-1}\| \nu(x) D \widetilde{u}\|^2_V
 \lesssim \tau^3 \|  \Op(\mu^{-1})\nu(x) \chi \widetilde{u}\|^2_V + \tau ^{-1}\| D^2 \Op(\mu^{-1})\nu(x) \chi  w\|^2_V.
\end{equation}

 \noindent
From Lemma \ref{lemma-remainder2},  we have
$   \Op(\mu^{-1})\nu \chi =  \nu \chi \Op(\mu^{-1}) +R_1$, where $R_1$ has a symbol in $S(\mu^{-2}\nu \lambda^{1\over2},g)$, and
$D^2 \Op(\mu^{-1})\nu \chi   =   \nu D^2   \chi  \Op(\mu^{-1})   +R_2$, where $R_2$ has a symbol in $S( \nu \lambda^{1\over2},g)$.
Then,  it follows from \eqref{eq: estimation by mu-1} that
\begin{align*}
\tau \| \nu(x) \widetilde{u}\|^2_V + \tau ^{-1}\| \nu(x) D \widetilde{u}\|^2_V
& \lesssim    \tau^3 \|\nu(x) \chi  \Op(\mu^{-1}) \widetilde{u}\|^2_V + \tau ^{-1}\| \nu(x) D^2 \chi  \Op(\mu^{-1}) \widetilde{u}\|^2_V + \lambda \| \nu(x) \widetilde{u}\|^2_V.
\end{align*}
For $\tau \ge \max\{{\cal C}\lambda,\,1\}$ with ${\cal C}$ large enough, one has the following result from the above inequality.
\begin{align}\label{eq: abs}
\tau \| \nu(x) \widetilde{u}\|^2_V + \tau ^{-1}\| \nu(x) D \widetilde{u}\|^2_V
& \lesssim    \tau^3 \|\nu(x) \chi  \Op(\mu^{-1}) \widetilde{u}\|^2_V + \tau ^{-1}\| \nu D^2 \chi  \Op(\mu^{-1}) \widetilde{u}\|^2_V .
\end{align}
Now, we
apply~\eqref{eq: second derivative estimate} to $v=\chi\Op(\mu^{-1})\widetilde{u}$ to have
 	\begin{equation*} \begin{array}{ll}
 &\tau^3\| \nu(x) \chi\Op(\mu^{-1})\widetilde{u}\|^2_V   +  \tau   \|\nu(x)  D  \chi\Op(\mu^{-1})\widetilde{u} \|^2_V+
\tau^{-1} \|\nu(x) D^2  \chi\Op(\mu^{-1})\widetilde{u} \|^2_V
\\ \noalign{\medskip}
 \lesssim & \| P_\varphi \chi\Op(\mu^{-1})\widetilde{u} \|^2_V  .
\end{array} \end{equation*}
Thus, combining this with \eqref{eq: abs} yields
\begin{align}\label{eq: abs1}
\tau \| \nu(x) \widetilde{u}\|^2_V + \tau ^{-1}\| \nu(x) D \widetilde{u}\|^2_V
& \lesssim    \| P_\varphi     \chi  \Op(\mu^{-1}) \widetilde{u} \|^2_V .
\end{align}

Finally, note that $P_\varphi$ has a symbol in $S(\nu\mu^2,g)$. Consequently, $P_\varphi     \chi  \Op(\mu^{-1})=  \Op(\mu^{-1})P_\varphi     \chi  +R$, where $R$ has a symbol in
$S(\nu\lambda^{1\over2},g)$. Then, we can deduce from \eqref{eq: abs1} that
$$
\tau \| \nu(x) \widetilde{u}\|^2_V + \tau ^{-1}\| \nu(x) D \widetilde{u}\|^2_V
\lesssim    \|\Op(\mu^{-1})  P_\varphi   \widetilde{u} \|^2_V+\lambda\| \nu \widetilde{u}\|^2_V .
$$
When $\tau\ge {\cal  C} \lambda$  with ${\cal C}$ large enough,   the error
term $\lambda \| \nu \widetilde{u}\|^2_V$ can be absorbed  by the left hand side. Consequently,
\begin{align}
	\label{eq: Carleman H-1 first form}
\tau \| \nu(x) \widetilde{u}\|^2_V + \tau ^{-1}\| \nu(x) D \widetilde{u}\|^2_V
& \lesssim    \|\Op(\mu^{-1})  P_\varphi   \widetilde{u} \|^2_V .
\end{align}
For $\widetilde{u}=e^{\tau\varphi}u$, it follows from \eqref{eq: Carleman H-1 first form} and similar argument as \eqref{eq:dc1c2} that
\begin{equation}
\label{eq: e0}
\begin{array}{lcl}
\tau \| \nu(x) e^{\tau\varphi}u\|^2_V + \tau ^{-1}\| \nu(x) e^{\tau\varphi}D  u\|^2_V &\lesssim  &   \tau \| \nu(x) e^{\tau\varphi}u\|^2_V + \tau ^{-1}\| \nu(x) D e^{\tau\varphi}u\|^2_V
\\ \noalign{\medskip} & \displaystyle
 \lesssim  &  \|\Op(\mu^{-1})  P_\varphi   e^{\tau\varphi}u \|^2_V .
\end{array}
\end{equation}
Obviously, one has that
\begin{align*}
P_\varphi \widetilde{u}= e^{\tau\varphi}P u=  e^{\tau\varphi} g_0+\sum_{j=1} ^d   e^{\tau\varphi  }\partial_{x_j} g_j=
e^{\tau\varphi} g_0+\sum_{j=1} ^d \big(   \partial_{x_j}( e^{\tau\varphi  } g_j ) -\tau e^{\tau\varphi  } g_j\partial_{x_j}\varphi   \big),
\end{align*}
which yields
\begin{align}
 \label{eq: e1}\|\Op(\mu^{-1})  P_\varphi   \widetilde{u} \|^2_V   \lesssim\sum\limits_{j=0}^d  \| e^{\tau \varphi}g_j \|^2_V.\end{align}
Hence,  we obtain Theorem~\ref{th: Carleman H-1} from \eqref{eq: e0} and \eqref{eq: e1}.
\end{proof}

\vskip 4mm

\begin{remark} \label{remark-1}
 Since $a(\cdot) $ is   nonnegative  and not identically null, there exists $\delta>0$ such that  $\{x\in\Omega\::\: a(x)> \delta\}\ne \emptyset$.
 We introduce several sets as follows.
 $$
 \begin{array}{l}
 W_0 = \Omega\setminus \supp\,a,
 \\ \noalign{\medskip}
 W_1= \overline{\Omega}\setminus {\cal O}\big(\supp\,a\big),
 \\ \noalign{\medskip}
W_2 = \Omega\setminus \big(\{x\in\Omega\::\: a(x)\ge\delta\}\cup {\cal O}(\Gamma)\big),
 \\ \noalign{\medskip}
    W_3=  \Omega \setminus \{x\in\Omega\::\: a(x)\ge  {\delta\over2}  \},
 \end{array}
$$
where $ {\cal O}(\Gamma)$ means the  neighborhood of $\Gamma.$

It is known that   there exists  a function
 $\psi\in\Con^\infty(\Omega) $   such that (\cite{Fursikov})
\begin{description}
\item[1)] $ \psi(x)=0$ for $x\in\partial\Omega$.
\item[2)] $\partial_n \psi(x)<0 $  for $x\in\partial\Omega$.
\item[3)] $\nabla \psi(x)\ne 0$ for $x\in \overline{\Omega\setminus\{x\in\Omega\::\: a(x)\ge  \delta \}}$. 
\end{description}
Let $\varphi=e^{\gamma \psi}$.  It follows  from Lemma \ref{lemma-sub-ellip} that  $\varphi $ satisfies the sub-ellipticity condition on
$x\in \Omega\setminus\{x\in\Omega\::\: a(x)\ge  \delta \}$ if $\gamma>0 $ is   sufficiently large.
Then, in Theorem \ref{th-Carleman} and \ref{th: Carleman H-1}, one can choose $V$ as   $\Omega\setminus\{x\in\Omega\::\: a(x)\ge  \delta \}$.
\end{remark}

The following result is a   classical Carleman estimate and  corresponding to   the Laplacian with Dirichlet boundary condition (\cite{lebau-robbiano},~Proposition 2). We shall use it to deal with the terms on $\Omega\setminus \supp\,a.$

\begin{lemma}
	\label{lemma: Carleman estimate up to boundary}
Suppose $\varphi$ is chosen as in Remark \ref{remark-1}.
Then, there exist ${\cal C} >0$, $\lambda_0>0$, such that for all
$u\in\Con^\infty(\overline{\Omega})$  satisfying  $\supp u \subset W_1$ and $u=0$ on $\Gamma$,  it holds
\begin{equation}\label{eq: classical Carleman estimate up tp the boundary}
\tau^3\|e^{\tau \varphi} u \|_{\Omega}^2+\tau\| e^{\tau \varphi} Du \|_{\Omega}^2\lesssim
 \| e^{\tau \varphi}(D^2-\lambda^2)u \|_{\Omega}^2 ,
\end{equation}
where  $\lambda\ge \lambda_0$ and
$\tau\ge  \max\{{\cal C}\lambda ,\,1\}$.
\end{lemma}



\begin{theorem}
\label{th: global Carleman up to boundary}
Suppose $\varphi$ is chosen as in Remark \ref{remark-1}.
Let $u\in\Con^\infty(\overline{\Omega})$  and
 satisfy
$$\begin{array}{lcl}\displaystyle
P(x,D,\lambda)u=f_0+\sum_{j=1}^d\partial_{x_j}f_j& \mbox{in} &\Omega,
\\ \noalign{\medskip}
u=0& \mbox{on} &\Gamma,   \end{array}
$$
where $f_j \in L^2(\Omega)$, $\supp f_0 \subset \Omega$
 and $\supp f_j \subset {\cal O}\big(\supp\,a\big)$ for $j=1,
\cdots , d. $
Then, there exist $  \tilde K>0$, $\lambda_0>0$, such that for all $\lambda\ge \lambda_0$ and
$\tau\ge  \max\{\tilde K|\lambda|^{5\over4},\,1\}$,
  it holds
\begin{align*}
 \tau\|e^{\tau \varphi} u \|_{W_3}^2+\tau^{-1}\| e^{\tau \varphi}  Du \|_{W_3}^2  \lesssim
   &  \sum_{j=0}^d  \| e^{\tau \varphi}f_j \|_\Omega^2
+ \lambda   \| e^{\tau \varphi}  u\|^2_{\{ x\in\Omega\,:\,  a(x)\ge \delta/2 \}} ,
\end{align*}
where the positive constant $\delta$ is defined as in Remark \ref{remark-1}.
\end{theorem}

\begin{proof}
Let $\chi_1,\;\chi_2\in\Con^\infty(\Omega)$ be non-negative and satisfy the following assumption
\begin{enumerate}
\item[{\rm (i)}] $0\le \chi_1,\;\chi_2\le1$; $\chi_1$  and $\chi_2$ are supported on $ W_1$ and $W_2$, respectively.
\item[{\rm (ii)}] $   \chi_1+\chi_2 \ge 1$ in $W_3$. In particular,
$\chi_1\equiv 1$ on   $ [{\cal O}(\d\Omega)\cap \Omega] \setminus  {\cal O}(\supp\,a)$, and $\chi_2\equiv1$ on ${\cal O}(\supp\,a) \setminus \{x\in \Omega\,:\, a(x) \ge {\delta \over2}\}.$
\end{enumerate}

 First, it is clear that $$P\chi_2u= \chi_2 f_0+ \sum\limits_{j=1}^d \big( \partial_{x_j}(\chi_2f_j)-f_j \partial_{x_j}\chi_2 \big) +[P,\chi_2]u.$$
Since $ [P,\chi_2]$ is a first order  operator,  we have there exist $a_0,\,a_1,\,\cdots,\,a_d$ and $b_0,\,b_1,\,\cdots,\,b_d$ such that  $$ [P,\chi_2] u=   \sum\limits_{j=0}^d a_j \partial_{x_j}u + \lambda\sum\limits_{j=0}^d  b_j\partial_{x_j} u  ,$$

\noindent
where $\supp\,a_j \subset \{x\in \Omega\,:\, {\delta\over 2} < a(x)<\delta\}\cup [\Omega\setminus({\cal O}(\supp\,a)\cup {\cal O}(\supp\,\partial\Omega))],\;\supp\, b_j \subset \{x\in \Omega\,:\, {\delta\over 2} < a(x)<\delta\} $   for $j=0,1,\cdots, d$.
Then,  applying Theorem~\ref{th: Carleman H-1} with $\chi_2u $ instead of $u$, $\Omega\setminus\{x\in\Omega\::\: a(x)\ge  \delta \}$ instead of V, we obtain
\begin{equation*}\label{eq: proof Carleman global 2}\begin{array}{ll}
&\displaystyle \tau\|( 1+\lambda^2a(x)^2 )^{1\over2} e^{\tau \varphi}\chi_2 u \|_{\Omega} ^2+\tau^{-1} \|  ( 1+\lambda^2a(x)^2 )^{1\over2} e^{\tau \varphi} D
\chi_2 u \|_{\Omega}^2
\\ \noalign{\medskip}
\lesssim &\displaystyle  \sum_{j=0}^d  \| e^{\tau \varphi}f_j \|_{\Omega}^2 +  \| e^{\tau \varphi} u\|^2_{\{x\in \Omega\,:\, {\delta\over 2} < a(x)<\delta\}\cup (\Omega\setminus\supp\,a)}
+  \lambda \| e^{\tau \varphi}u\|^2_{\{x\in \Omega\,:\, {\delta\over 2} < a(x)<\delta\} }.
\end{array}\end{equation*}
Consequently, due to $( 1+\lambda^2a(x)^2 )\ge 1
$, $ \tau\ge \{{\cal C}\lambda,\, 1\}$ and $\lambda\ge \lambda_0$,
we obtain
\begin{equation}\label{eq: proof Carleman global 2}\begin{array}{ll}
&\displaystyle \tau\|( 1+\lambda^2a(x)^2 )^{1\over2} e^{\tau \varphi}\chi_2 u \|_{\Omega} ^2+\tau^{-1} \|  ( 1+\lambda^2a(x)^2 )^{1\over2} e^{\tau \varphi} D
\chi_2 u \|_{\Omega}^2
\\ \noalign{\medskip}
\lesssim &\displaystyle  \sum_{j=0}^d  \| e^{\tau \varphi}f_j \|_{\Omega}^2 +  \| e^{\tau \varphi} u\|^2_{ \Omega\setminus\supp\,a }
+  \lambda \| e^{\tau \varphi}u\|^2_{\{x\in \Omega\,:\, a(x)\ge {\delta\over 2}\} }.
\end{array}\end{equation}
%

On the other hand, by using  $\chi_1u $ instead of $u$ in  Lemma~\ref{lemma: Carleman estimate up to boundary}, we   obtain
\begin{align}
 	\label{eq: proof Carleman global 0}\tau^3\|e^{\tau \varphi} \chi_1u \| _{\Omega}^2+\tau \| e^{\tau \varphi} D\chi_1u \|_{\Omega}^2\lesssim
 \| e^{\tau \varphi} (P - i\lambda D a(x) D )\chi_1u \|_{\Omega}^2  .
\end{align}
 Since $\chi_1\partial_{x_j} f_j=0$ for $j=1,\cdots,d,$
 we have $$P\chi_1u=\chi_1 f_0+ [P,\,\chi_1]u.$$
Therefore, combining these with \eqref{eq: proof Carleman global 0} and  ${\cal O}(\supp \,a )\cap \supp\, \chi_1 =\emptyset$ yields
 \begin{equation}
 	\label{eq: proof Carleman global 1}
 \tau\|e^{\tau \varphi} \chi_1u \| _{\Omega}^2+\tau^{-1}\| e^{\tau \varphi} D\chi_1u \|_{\Omega}^2\;\lesssim\;
\tau^{-2} \big(\| e^{\tau \varphi}   f_0 \|_{\Omega}^2 +  \| e^{\tau \varphi}  [P,\,\chi_1]u\|_{\Omega}^2 \big) .
  \end{equation}

\noindent
Adding up ~\eqref{eq: proof Carleman global 1} and ~\eqref{eq: proof Carleman global 2}, using $( 1+\lambda^2a(x)^2 )\ge 1
$ and $\tau\ge  1$,
we obtain
\begin{equation*}\begin{array}{ll} &\displaystyle
\tau\|e^{\tau \varphi}( \chi_1+\chi_2) u \|_\Omega^2+\tau^{-1}\| e^{\tau \varphi} D(\chi_1+\chi_2)u \|_\Omega^2
 \\ \noalign{\medskip}\lesssim & \displaystyle
 \sum_{j=0}^d  \| e^{\tau \varphi}f_j \|_\Omega^2
  +  \| e^{\tau \varphi} u\|^2_{ \Omega\setminus\supp\,a }
+  \lambda \| e^{\tau \varphi}u\|^2_{\{x\in \Omega\,:\,   a(x)\ge {\delta\over2}\} }
 +\tau^{-2} \| e^{\tau \varphi}  [P,\,\chi_1]u\|_{\Omega}^2   .
\end{array}
\end{equation*}

\noindent
Note that  $ [P,\chi_1] $ is a first order operator and supported on $\supp\,\chi_1 \setminus\{x\in \Omega\,:\, \chi_1  \equiv 1\}$, which is an subset of $\{x\in \Omega\,:\, \chi_1+\chi_2=1\}$. Then,  for $\tau$ sufficiently large,  we have
\begin{align*}
 \tau\|e^{\tau \varphi}( \chi_1+\chi_2) u \|_\Omega ^2+\tau^{-1}\| e^{\tau \varphi} (\chi_1+\chi_2) Du \|_\Omega^2 & \lesssim
 \sum_{j=0}^d  \| e^{\tau \varphi}f_j \|_\Omega^2
+ \lambda \| e^{\tau \varphi}  u\|^2_{ \{ x\in\Omega\,:\,a(x)\ge {\delta\over2} \} }.
\end{align*}
Then as $1\le \chi_1+\chi_2\le 2 $  on $W_3$, we obtain the statement of
Theorem~\ref{th: global Carleman up to boundary}.
 \end{proof}

\section{Resolvent estimate}
\label{section: resolvent estimate}
\setcounter{equation}{0} \setcounter{theorem}{0}

In this section, we shall  prove   the main result.
From the results in \cite{Batty-Duyckaerts, burq, Duyckaerts},   the logarithmic decay of the energy
in Theorem \ref{th2} can be obtained through the following theorem.
\begin{theorem}\label{th1}
Suppose  conditions in Theorem \ref{th2} hold. Then, for every $\lambda \in \mathbb{R}$ with $|\lambda|$  large enough,
 there exists $C>0 $ such that
\begin{equation}
\label{th-resolvent}
\Big\|({\cal A} -i\lambda)^{-1}\Big\|_{{\cal L}({\cal H})} \lesssim e^{C\,|\,\lambda\,|\,^{5\over4}}.
\end{equation}
\end{theorem}

%
%

\indent
Let $\lambda$ be a real number such that $\,|\,\lambda\,|\,$ is large enough.
Consider the resolvent equation:
\begin{equation}
\label{eq: resolvent1}
F = ({\cal A}- i\lambda)Y,  \quad \;  \mbox{where   }  \; Y = \big( y_1 ,\,y_2 \big) \in D({\cal A}), \quad F=\big( f_1 ,\,f_2 \big)\in \cal{H},
\end{equation}
which yields
\begin{equation}
\label{eq: resolvent2}
\left\{
\begin{array}{ll}
 \Div\,(\nabla y_1 + a \nabla y_2)   +\lambda^2 y_1  = i\lambda f_1+ f_2, & \;  \mbox{  in }  \; \Omega,
\\ \noalign{\medskip}
 y_2 = i\lambda y_1+f_1 & \;  \mbox{  in }  \; \Omega,
\\ \noalign{\medskip}
y_1\,|\,_{\Gamma} =0
\end{array}
\right.
\end{equation}

\noindent
 In what follows, we shall prove that there exists a positive constant $C$ such that
\begin{align*}
\| (y_1,\,y_2) \|_{\cal H}  \lesssim
e^{C\,|\,\lambda\,|\,^{5\over4}} \| (f_1,\,f_2 )\|_{\cal H}.
\end{align*}

First, let $\eta>0$ and  $\chi\in \Con_0^\infty(\Omega)$ be real valued function  such that $\supp \,\chi = \{x\in\Omega\,:\, a(x)>\eta\})$
and $\chi \equiv 1$ in $\{x\in \Omega\,:\, a(x)>2\eta\}$. The following lemma is helpful.

\begin{lemma}
	\label{lem: estimate L2 norm by derivative}
For $y_1,\; f_1,\; f_2$  satisfying \eqref{eq: resolvent2} and $|\lambda|$ large enough, it holds
\begin{enumerate}
\item[{\rm (i)}] $\| \chi y_1 \|_\Omega \lesssim \| \nabla y_1\|_{\{x\in\Omega\,:\, a(x)>\eta\}}+\| f_1\|_{H^1(\Omega)}+\| f_2\|_{\Omega}.
$
\item[{\rm (ii)}] $\displaystyle \int_{\Omega} a(x) |\nabla y_1|^2   dx \lesssim \| f_1\|_{H^1(\Omega)}^2
+ (\| f_1\|_\Omega+\| f_2\|_\Omega)\| y_1\|_\Omega .$
\end{enumerate}
 \end{lemma}
\begin{proof}
(i)~ Multiplying the first equation in \eqref{eq: resolvent2} by $\chi^2 \overline y_1$  and using $y_2 = i\lambda y_1+f_1 $,  we obtain,
\begin{equation}\label{eq: q1}\begin{array}{lcl}
\lambda^2 \|\chi y_1\|_\Omega^2 & =& \displaystyle
 \int_\Omega \nabla y_1\cdot \nabla( \chi^2  \overline y_1) dx + i\lambda  \int_\Omega a(x) \nabla y_1\cdot \nabla(   \chi^2  \overline y_1) dx
\\ \noalign{\medskip} && \displaystyle
      + i\lambda \int_\Omega  f_1 \chi^2 \overline y_1 dx + \int_\Omega  f_2\chi^2 \overline y_1 dx
+\int_\Omega a(x) \nabla f_1\cdot \nabla(\chi^2 \overline y_1)dx.
\end{array}\end{equation}
 Since $ \nabla( \chi^2  \overline y_1) = \chi^2 \nabla  \overline y_1   + 2 \overline y_1\chi \nabla\chi$, we have
\begin{equation}\label{eq: q3}\begin{array}{ll}\displaystyle
\,\Big|\,  \int_\Omega a(x) \nabla y_1 \cdot\nabla(   \chi^2  \overline y_1) dx    \,\Big|\,& \displaystyle
\lesssim   \| \nabla y_1\|_{\{x\in\Omega\,:\, a(x)>\eta\}}^2
+\| \chi y_1 \|_\Omega \| \nabla y_1\|_{\{x\in\Omega\,:\, a(x)>\eta\}}
\\ \noalign{\medskip} &\le  \displaystyle
 2\| \nabla y_1\|_{\{x\in\Omega\,:\, a(x)>\eta\}}^2
+{1\over4}\| \chi y_1 \|_\Omega^2  .
\end{array}\end{equation}

\noindent
By the similar argument, we can deal with the rest terms on the right hand side of \eqref{eq: q1}. Combining these with  \eqref{eq: q1},  \eqref{eq: q3} yields
\begin{equation}\label{eq: q2}\begin{array}{lcl}
\lambda^2 \|\chi y_1\|_\Omega^2&\lesssim &\displaystyle {1\over 2}\lambda^2 \| \chi y_1 \|_\Omega^2+ ( |\lambda|+1) \| \nabla y_1\|_{\{x\in\Omega\,:\, a(x)>\eta\}}^2
\\ \noalign{\medskip}& & \displaystyle
   + \| f_1\|_\Omega^2+\lambda^{-2}\| f_2\|_\Omega^2+ (\lambda^{-2}+1) \| \nabla f_1\|^2_{H^1(\Omega)}.
\end{array}\end{equation}
Then, (i) is reached.

(ii)~ Multiplying the first equation in \eqref{eq: resolvent2} by $  \overline y_1$  and using $y_2 = i\lambda y_1+f_1 $,  we obtain,
\begin{align}\label{eq: q41}
\int_{\Omega}[-(i\lambda a(x) +1) |\nabla y_1|^2    +\lambda^2 |y_1|^2 ]dx
 =  \int_{\Omega}[a(x)\nabla f_1\cdot\nabla\overline{y}_1+( i\lambda f_1+  f_2) \overline{y}_1]dx.
\end{align}
Taking the imaginary part of \eqref{eq: q41} yields
\begin{align}\label{eq: q5}
\int_{\Omega} a(x) |\nabla y_1|^2   dx = - \lambda^{-1} \Im \int_{\Omega} [a(x)\nabla f_1\cdot\nabla\overline{y_1}
+( i\lambda f_1+  f_2) \overline{y_1}]dx.
\end{align}
Thus, by the Cauchy-Schwarz inequality, one can conclude from \eqref{eq: q5} that
\begin{equation*}\begin{array}{ll}\displaystyle
\int_{\Omega} a(x) |\nabla y_1|^2   dx & \le \displaystyle |\lambda\,|^{-1}\,\Big( \int_{\Omega} a(x) |\nabla y_1|^2   dx\Big)^{1\over2}
\| f_1\|_{H^1(\Omega)}+  ( \| f_1\|_{\Omega}+|\lambda\,|^{-1}\,\| f_2\|_{\Omega})\| y_1\|_{\Omega}
\\ \noalign{\medskip} & \le \displaystyle
{1\over4} \int_{\Omega} a(x) |\nabla y_1|^2   dx
+
|\lambda\,|^{-2}\,
\| f_1\|^2_{H^1(\Omega)}+ ( \| f_1\|_{\Omega}+|\lambda\,|^{-1}\,\| f_2\|_{\Omega})\| y_1\|_{\Omega}.
\end{array}
\end{equation*}
The proof of lemma is finished.
\end{proof}

{\bf Proof  of Theorem~\ref{th1}.}
Due to Lemma~\ref{lem: estimate L2 norm by derivative} (i) and (ii), we have
\begin{align}\label{eq: q6}
\| y_1 \|_{\{x\in\Omega\,:\, a(x)>2\eta\}}^2 \lesssim \| \nabla y_1\|_{ \{x\in\Omega\,:\, a(x)>\eta\}}^2
+\| f_1\|_{H^1(\Omega)}^2+\| f_2\|_{\Omega}^2,
\end{align}
and\begin{align}\label{eq: q7}
\|  \nabla y_1 \|_{ \{x\in\Omega\,:\, a(x)>\eta\}}^2  \lesssim \| f_1\|_{H^1(\Omega)}^2
+ (\| f_1\|_{\Omega}+\| f_2\|_{\Omega})\| y_1\|_{\Omega}   .
\end{align}
Combining \eqref{eq: q6} and \eqref{eq: q7} yields
\begin{align}
	\label{est: estimate H1 norm on a positive}
\|  y_1 \|_{H^1(\{x\in\Omega\,:\, a(x)>2\eta\})}^2  \lesssim \| f_1\|_{H^1(\Omega)}^2+\| f_2\|_{L^2(\Omega)}^2
+ (\| f_1\|_{\Omega}+\| f_2\|_{\Omega})\| y_1\|_{\Omega}  .
\end{align}

\noindent
On the other hand, by \eqref{eq: resolvent2}, one has that
  $y_1$ satisfies
\begin{align}\label{eq: y1}
 \Div\,(\nabla y_1 +i\lambda a(x) \nabla y_1)   +\lambda^2 y_1  = i\lambda f_1+ f_2-\Div (a(x) \nabla f_1) .
\end{align}
Then, applying Theorem~\ref{th: global Carleman up to boundary} to $y_1$ satisfying \eqref{eq: y1},
we obtain
\begin{equation*}\begin{array}{ll}
 &\displaystyle \tau\|e^{\tau \varphi} y_1 \|_{\Omega}^2+\tau^{-1}\| e^{\tau \varphi} \nabla y_1 \|_{\Omega}^2
 \\ \noalign{\medskip}  \lesssim&\displaystyle
 \lambda^2  \| e^{\tau \varphi}  f_1  \|_{\Omega}^2 + \| e^{\tau \varphi}f_2 \|_{\Omega}^2 +
  \sum\limits_{j=1}^d  \| e^{\tau \varphi}a\,\d_{x_j}f_1 \|_{\Omega}^2
 \\ \noalign{\medskip} & \displaystyle
+ (\lambda+\tau ) \| e^{\tau \varphi}  y_1\|^2_{   \{ x\in\Omega\,:\,  a(x)\ge \delta/2 \}  }
+\tau^{-1}\| e^{\tau \varphi} \nabla y_1\|^2_{   \{ x\in\Omega\,:\,  a(x)\ge \delta/2 \}  }.
\end{array}
\end{equation*}
Let $c_1=\min _{x\in\Omega}\varphi(x)$ and $c_2=\max_{x\in\Omega}\varphi(x)$. We conclude from the above inequality and $\tau\ge  \max\{\widetilde K |\lambda|^{5\over4},\,1\}$ that
\begin{equation}\label{eq: q4}\begin{array}{ll}
&  \quad \tau  e^{2 c_1\tau }\| y_1 \|_{\Omega} ^2+\tau^{-1} e^{2 c_1\tau } \|   \nabla y_1 \|_{\Omega}^2
\\ \noalign{\medskip} \lesssim&
 \lambda^2 e^{2 c_2 \tau }   \|   f_1  \|_{H^1(\Omega)}^2 + e^{2 c_2 \tau }    \|f_2 \|_{\Omega}^2
+ \tau e^{2 c_2 \tau }     \|   y_1\|^2_{H^1(  \{ x\in\Omega\,:\,  a(x)\ge \delta/2 \}  )} .
\end{array}\end{equation}
 Setting $\eta={\delta\over 8}$ and substituting \eqref{est: estimate H1 norm on a positive}   into \eqref{eq: q4}, we obtain
\begin{align*}
\tau e^{2 c_1\tau } \| y_1 \|_{\Omega} ^2+\tau^{-1} e^{2 c_1\tau } \| \nabla y_1 \|_{\Omega}^2  \lesssim
 \tau^2 e^{2 c_2 \tau }   \|   f_1  \|_{H^1(\Omega)}^2 + \tau e^{2 c_2 \tau }    \|f_2 \|_{\Omega}^2
+\tau e^{2 c_2 \tau }  (\| f_1\|_{\Omega}+\| f_2\|_{\Omega})\| y_1\|_{\Omega}     .
\end{align*}
Let $c_3=2(c_2-c_1)+1$.  For  $\tau\ge  \max\{\widetilde K |\lambda|^{5\over4},\,1\}$, one has
\begin{align*}
\| y_1 \|_{\Omega} ^2+ \|  \nabla y_1 \|_{\Omega}^2  \lesssim
 e^{c_3 \tau }   \|   f_1  \|_{H^1(\Omega)}^2 + e^{ c_3 \tau }    \|f_2 \|_{\Omega}^2
+ e^{ c_3 \tau }  (\| f_1\|_{\Omega}+\| f_2\|_{\Omega})\| y_1\|_{\Omega}     .
\end{align*}
For any $\varepsilon>0$, using $ e^{ c_3 \tau }  (\| f_1\|_{\Omega}+\| f_2\|_{\Omega})\| y_1\|_{\Omega}    \le \varepsilon \| y_1\|_{\Omega}^2 +{1\over4} \varepsilon^{-1} e^{2 c_3 \tau }
(\| f_1\|_{\Omega}+\| f_2\|_{\Omega})^2$ in the above estimate, we conclude that
\begin{align*}
\| y_1 \|_{\Omega} ^2+ \|   \nabla y_1 \|_{\Omega}^2  \lesssim
 e^{2c_3 \tau } \big(  \|   f_1  \|_{H^1(\Omega)}^2 +    \|f_2 \|_{\Omega}^2\big)    ,
\end{align*}
which gives the desired result.
\hfill $\Box$

%
%
\appendix

%
%
%
%

\end{document}